\documentclass[a4paper]{amsart}

\usepackage{amsmath}
\usepackage{amssymb}
\usepackage{eucal}
\usepackage{pstricks}
\usepackage{pstricks-add}
\psset{nodesep=2pt,labelsep=1pt,arrowscale=1.2}
\usepackage{enumerate}
\usepackage{url}
\usepackage[normalem]{ulem}

\theoremstyle{plain}
\newtheorem{theorem}{Theorem}[section]
\newtheorem{proposition}[theorem]{Proposition}
\newtheorem{lemma}[theorem]{Lemma}
\newtheorem{corollary}[theorem]{Corollary}
\theoremstyle{definition}
\newtheorem*{definition}{Definition}
\theoremstyle{remark}
\newtheorem*{remark}{Remark}

\newcommand{\injto}{\hookrightarrow}
\newcommand{\End}{\operatorname{End}}
\newcommand{\Aut}{\operatorname{Aut}}
\newcommand{\Age}{\operatorname{Age}}
\newcommand{\pTp}{\operatorname{Tp}^{\exists_1^+}}
\newcommand{\pTh}{\operatorname{Th}^{\exists_1^+}}
\newcommand{\Tp}{\operatorname{Tp}}
\newcommand{\Th}{\operatorname{Th}}
\newcommand{\cAut}{\operatorname{cAut}}
\newcommand{\wAut}{\operatorname{wAut}}
\newcommand{\sAut}{\operatorname{sAut}}
\newcommand{\Col}{\operatorname{Col}}
\newcommand{\bA}{\mathbf{A}}
\newcommand{\bB}{\mathbf{B}}
\newcommand{\bC}{\mathbf{C}}
\newcommand{\bD}{\mathbf{D}}
\newcommand{\bF}{\mathbf{F}}
\newcommand{\bK}{\mathbf{K}}
\newcommand{\bL}{\mathbf{L}}
\newcommand{\bT}{\mathbf{T}}
\newcommand{\bU}{\mathbf{U}}
\newcommand{\cA}{\mathcal{A}}
\newcommand{\cB}{\mathcal{B}}
\newcommand{\cC}{\mathcal{C}}
\newcommand{\cF}{{\mathcal{F}}}
\newcommand{\cK}{\mathcal{K}}
\newcommand{\cL}{\mathcal{L}}
\newcommand{\cU}{\mathcal{U}}
\newcommand{\ba}{\bar{a}}
\newcommand{\bb}{\bar{b}}
\newcommand{\bc}{\bar{c}}
\newcommand{\bd}{\bar{d}}
\newcommand{\bu}{\bar{u}}
\newcommand{\bx}{\bar{x}}

\newcommand{\fA}{\mathfrak{A}}
\newcommand{\fB}{\mathfrak{B}}
\newcommand{\fC}{\mathfrak{C}}
\newcommand{\fD}{\mathfrak{D}}

\newcommand{\CSP}{\operatorname{CSP}}
\newcommand{\HE}{{\mathcal{E}}}
\newcommand{\Forb}{{\mathit{Forb}}}
\newcommand{\KLF}{{\mathcal{K}_{\mathcal{L}\mathcal{F}}}}
\newcommand{\Fraisse}{Fra\"\i{}ss\'e}

\title[Universal homogeneous constraint structures\dots]{Universal homogeneous constraint structures and the hom-equivalence classes of weakly oligomorphic structures}
\author{Christian Pech}
\address{Christian Pech\\L\"o\ss nitzgrundstra\ss e 47\\01445 Radebeul\\Germany}
\email{cpech@freenet.de}
\author{Maja Pech}
\address{Maja Pech\\Department of Mathematics and Informatics\\University of Novi Sad\\Trg Dositeja Obradovi\'ca 4\\ 21000 Novi Sad\\Serbia}
\thanks{supported by the Ministry of Education and Science of the Republic of Serbia through Grant No.174018, and by a grant (Contract 114-451-1901/2011) of the the Secretariat of Science and Technological Development of the Autonomous Province of Vojvodina.}
\email{maja@dmi.uns.ac.rs}
\begin{document}

\begin{abstract}
We derive a new sufficient condition for the existence of $\aleph_0$-categorical universal structures in classes of relational structures with constraints, augmenting results by Cherlin, Shelah, Chi \cite{CheSheChi99}, and Hubi\v{c}ka and Ne\v{s}et\v{r}il \cite{HubNes09b}. 

Using this result we show that  the  hom-equivalence class of any countable weakly oligomorphic structure  has up to isomorphism a unique model-complete smallest and  greatest element, both of which are $\aleph_0$-categorical. 

As the main tool we introduce the category of constraint structures, show the existence of  universal homogeneous objects, and study their automorphism groups.

All constructions rest on a category-theoretic version of \Fraisse's Theorem due to Droste and G\"obel. We derive sufficient conditions for a comma category to contain a universal homogeneous object. 

This research is motivated by the observation that all countable models of the theory of a weakly oligomorphic structure are hom-equivalent---a result akin to (part of) the Ryll-Nardzewski Theorem.
\end{abstract}
\maketitle

\section*{Introduction}
Weak oligomorphy is a natural weakening of the notion of oligomorphy. The latter concept was coined by Peter Cameron in the 1970$^{\text{th}}$ and it turned out to be fundamental in several fields of mathematics \cite{Cam90}. Not only many combinatorial enumeration problems have a natural encoding by oligomorphic permutation groups, but the Engeler-Ryll-Nardzewski-Svenonius Theorem links oligomorphic permutation groups to $\aleph_0$-categorical structures and hence to model theory \cite{Hod97}. Both notions, oligomorphy and $\aleph_0$-categoricity, are in turn closely related to the concept of (ultra-) homogeneity and thus with the theory of \Fraisse-limits \cite{Fra53,Mac11}. 

In their seminal paper \cite{CamNes06}, Peter Cameron and Jaroslav Ne\v{s}et\v{r}il introduced several variations to the concept of homogeneity, one of them being homomorphism-homogeneity---saying that every homomorphism between finitely generated substructures of a given structure extends to an endomorphism of that structure. The relevance of this notion in the theory of transformation monoids on countable sets was realized quickly \cite{Dol12,Dol12b,MPPhD,AUpaper,Pon05}. Also a classification theory for homomorphism-homogeneous structures emerged quickly \cite{CamLoc10,DolMas11,IliMasRaj08,JunMas12,Mas07,Mas12,MasNenSko11} and classes of high complexity of finite homomorphism-homogeneous structures were discovered \cite{Mas12b,RusSch10}.

Weak oligomorphy is a phenomenon that arrises naturally in the context of homomorphism-homogeneity. 
A countable relational structure is  weakly oligomorphic if  its endomorphism monoid is oligomorphic, i.e., it has of every arity only finitely many invariant relations. It is not hard to see that every homomorphism-homogeneous relational structure over a finite signature is weakly oligomorphic.  It is less obvious that every weakly oligomorphic relational structure has a positive existential expansion that is homomorphism-homogeneous (cf. \cite{MasPec11,MPPhD,AUpaper}). Clearly, every oligomorphic structure is weakly oligomorphic, but the reverse does not hold, in general. 

It turns out that weak oligomorphy is the key for the development of a model-theory of homomorphism-homogeneous structures running in parallel to the model theory of homogeneous structures \cite{MasPec11,AUpaper}. However, we think that in order to be sustainable, the theory of weak oligomorphy and homomorphism-homogeneity has not only to exist for its own sake but it has to augment the branches of mathematics from whom it was inspired. First steps into this direction were undertaken in \cite{AgeHHnew}, where it was shown that every weakly oligomorphic structure is homomorphism-equivalent to an oligomorphic substructure. This, e.g., has the consequence that every constraint satisfaction problem with a weakly oligomorphic template is equivalent to one with an $\aleph_0$-categorical template. Another consequence is that a set of positive existential propositions is the complete positive existential theory of a weakly oligomorphic structure if and only if it is the positive existential part of an $\aleph_0$-categorical theory. 

In this paper we are going  further into the direction of creating cross-links between the theory of weakly oligomorphic structures and the theory of oligomorphic structures.

Our foremost interest belongs  to the hom-equivalence classes of weakly oligomorphic structures.  In the Section~\ref{s2} we give some motivations, why we are interested in the structure of hom-equivalence classes of weakly oligomorphic structures. While the Ryll-Nardzewski Theorem says that any countable oligomorphic structure  is up to isomorphism determined by its first order theory, we will show that countable weakly oligomorphic structures are up to homomorphism equivalence determined by their first order theory. This property is called weak $\aleph_0$-categoricity. 

In Section~\ref{s3}, we describe the extremal  elements of hom-equivalence classes of weakly oligomorphic structures with respect to inclusion. We show that every such class has a canonical largest and smallest element both of which are oligomorphic and model-complete (i.e. finite or $\aleph_0$-categorical). In fact we prove a more general result about the existence of universal elements in given classes of countable relational structures that augments results from \cite{CheSheChi99,Hub10,HubNes09}. 

In Section~\ref{s4} we introduce and study constraint structures. These are basically solutions to instances of constraint satisfaction problems. More precisely, for a template $\bT$ a $\bT$-colored structure is a pair $(\bA,a)$ such that $a$ is a homomorphism from $\bA$ to $\bT$. We introduce strong and weak homomorphisms between constraint structures and show the existence of universal homogeneous $\bT$-colored structures. This is done using a categorical variant of \Fraisse's theorem due to Droste and G\"obel \cite{DroGoe92,DroGoe93}. Universal homogeneous constraint structures are used for proving the existence result of universal elements from Section~\ref{s3} and thus, ultimately they give rise to universal elements in hom-equivalence classes of countable relational structures. We give conditions under which the strong or weak automorphism group of a universal homogeneous constraint structure is oligomorphic.

Universal homogeneous constraint structures are closely related to monotone free amalgamation classes (indeed, if the universe in which they live is monotone, then the age of a universal homogeneous constraint structure is monotone and free). For relational structures these two properties have strong consequences. From one hand, given that the relational signature is finite,  it means that the automorphism group of the \Fraisse-limit has the small index property \cite{Her98,Mac11}, and hence the \Fraisse-limit of the class can be reconstructed up to first order interpretation from its automorphism group (considered as abstract group). From the other hand, for any monotone free \Fraisse-class $\cC$, the class of linear ordered structures $(\cC,\prec)$ is a Ramsey-class \cite{HubNes09,NesRoe77}. In Section~\ref{s5} we give sufficient conditions for the (strong) automorphism group of a universal homogeneous constraint structure to have the small index property, the Bergman-property, and uncountable cofinality (cf. \cite{Ber06,DroGoe05,KecRos07}).

Finally, in Section~\ref{s6}, we develop the necessary categorical framework  for the construction of universal homogeneous constraint structures. In particular, we give sufficient conditions under which a comma-category contains a universal homogeneous object. This is the technical backbone of our approach to universal structures. Its high level of abstraction  makes it perhaps seem like a tool too big  for the task at hand. However, due to its categorical nature it has potentially a much wider field of application. This, in our opinion, makes it  worthwhile  to be included into this paper.

\section{Preliminaries}
This paper deals (mainly) with relational structures. Here, under a relational structure we understand a model-theoretic structure without operations and constants. Accordingly, a relational signature is a model-theoretic signature without operational- or constant-symbols. Whenever we do not state otherwise, a relational signature can contain any number of relational symbols. Relational structures will be denoted like $\bA,\bB,\bC,\dots$. There carriers are denoted like $A,B,C,\dots$. Tuples are denoted like $\ba,\bb,\bc$, and usually $a_i$ will denote the $i$-th coordinate of $\ba$, etc.

As usual, a homomorphism between  relational structures is a
function between the carriers that preserves all relations.

If $f:\bA\to\bB$, then we call $\bA$ the \emph{domain} of $f$ and $\bB$, the \emph{codomain}. Moreover, the
structure induced by
$f(A)$ is called the \emph{image} of $f$. 

\emph{Epimorphisms} are surjective homomorphisms and \emph{monomorphisms} are
injective homomorphisms. Isomorphisms are bijective homomorphisms
whose inverse is a homomorphism, too. \emph{Embeddings} are monomorphisms that not only
preserve relations but also reflect them. That is, a monomorphism is
an embedding if and only if it is an isomorphism to its image.

For classes $\cA$, $\cB$ of relational structures we write $\cA\to\cB$ if for every $\bA\in\cA$ there exists a $\bB\in\cB$ and a homomorphism $f:\bA\to\bB$. Instead of $\{\bA\}\to\cB$ we write $\bA\to\cB$ and instead of $\cA\to \{\bB\}$ we write $\cA\to\bB$. An important special case is the notion $\bA\to\bB$ which means that there is a homomorphism from $\bA$ to $\bB$. If $\bA\to\bB$ and
$\bB\to\bA$, then we call $\bA$ and $\bB$ \emph{homomorphism-equivalent}. 

If $\cA$ is a class of relational structures over a signature $R$, then by $(\cA,\to)$ and $(\cA,\injto)$ we will denote the categories of objects from $\cA$ with homomorphisms or embeddings as morphisms, respectively.

On some occasions we will need the notion of the \emph{Gaifman-graph} of a relational structure $\bA$. This is a simple graph whose vertex set is $A$, such that two vertices are joint by an edge whenever they occur together in a tuple from one of the basic relations of $\bA$. We will denote the Gaifman-graph of $\bA$ by $\Gamma_\bA$. We call $\bA$ \emph{connected} if $\Gamma_\bA$ is connected and we call $\bA$ tight if its Gaifman-graph is complete.

\subsection*{Ages of relational structures.} The \emph{age} of a relational structure $\bA$ is the class of all finite relational structures that embed into $\bA$ (it is denoted by $\Age(\bA)$).

Let $\cC$ be a class of finite relational structures over the same signature. We say that $\cC$ has the \emph{Joint embedding property (JEP)} if whenever $\bA,\bB\in \cC$, then there exists a $\bC\in\cC$ such that both $\bA$ and $\bB$ are embeddable in $\bC$. 
Moreover, $\cC$ has the \emph{Hereditary property (HP)} if whenever $\bA\in \cC$, and $\bB< \bA$,    then $\bB$ is isomorphic to some $\bC\in \cC$.
We say that $\cC$ has the \emph{amalgamation property (AP)} if whenever $\bA, \bB_1,\bB_2 \in\cC$, and $f_1:\bA\to \bB_1$ and $f_2:\bA\to\bB_2$ are embeddings, then there are $\bC\in \cC$, and embeddings $g_1:\bB_1\to \bC$ and $g_2:\bB_2\to\cC$    such that $g_1\circ f_1=g_2\circ f_2$.  Finally, $\cC$ has the \emph{homo-amalgamation property (HAP)} if whenever $\bA, \bB_1,\bB_2 \in \cC$, $f_1:\bA\to \bB_1$ is a homomorphism, and $f_2:\bA\to \bB_2$ is an embedding, then there are $\bC\in \cC$, an embedding    $g_1:\bB_1\to \bC$ , and a homomorphism $g_2:\bB_2\to\cC$ such that $g_1\circ f_1=g_2\circ f_2$.   

A basic theorem by Roland \Fraisse{} states that a class of finite
structures of the same type is the age of a countable structure if and
only if 
\begin{enumerate}
\item it has only countably many isomorphism types,
\item it is isomorphism-closed,
\item it has the  (HP) and the  (JEP).
\end{enumerate}
A class of finite structures with these properties will be called an \emph{age}. If $\cC$ is an age, then by $\overline{\cC}$ we denote the class of all countable structures whose age is contained in $\cC$. 

\subsection*{Constraint satisfaction problems.} For a relational structure $\bT$, with $\CSP(\bT)$ we denote the class of all those finite relational structures $\bA$ for which $\bA\to\bT$. The notion $\CSP$ comes from theoretical computer science and reads \emph{constraint satisfaction problem}. It refers to the problem to decide whether for a given finite relational structure $\bA$ there exists a homomorphism into $\bT$. If $\bT$ is finite over a finite signature, then $\CSP(\bT)$ is called a finite constraint satisfaction problem. Clearly, finite CSPs are in NP. It is conjectured that every finite constraint satisfaction problem is either in P or it is NP-complete (originally, this was stated by Feder and Vardi in \cite{FedVar99}). Such  dichotomy-conjectures  are very intensively studied in theoretical computer science and in universal algebra. It is impossible to give a reasonable bibliographic overview of this topic but a few important step-stones are \cite{BarKoz09,BarKozNiv08,BerIdzMarMcKVal10,Bod08,Bod12,BodNes06,BodPin11,BulKroJea05,Bul06,HelNes90,Sch78}. In this paper we are not concerned with complexity aspects of CSPs. Rather, CSPs appear in the context of our research as analogues of ages. Yet, our structural results may be of some interest also for the research on the complexity -classification of constraint satisfaction problems. E.g., a consequence of Theorem~\ref{univ-struc} is that every CSP over  a finite or $\aleph_0$-categorical template is equivalent to a substructure-problem and to a weak substructure problem of an $\aleph_0$-categorical template. 

We will use the operator $\CSP$ for all kinds of relational structures over arbitrary relational signatures. It was observed by Feder (cf. also \cite[Lemma 1]{Bod08}) that a class of relational structures is the constraint satisfaction problem of a structure if and only if it is closed with respect disjoint union and inverse homomorphisms (a class $\cC$ of finite relational structures is closed with respect to inverse homomorphisms if whenever $\bB\in\cC$ and $\bA$ is finite such that $\bA\to\bB$, then also $\bA\in\cC$).

Classes of finite relational structures that are closed with respect to disjoint union and inverse homomorphism will just be called \emph{constraint satisfaction problems} or CSPs, for short. Note that every CSP over a countable signature is an age, too.

\subsection*{Weakly oligomorphic structures.} 

In \cite{Cam90}, Peter Cameron defined the notion of oligomorphic permutation groups. Recall that a permutation group is called \emph{oligomorphic} if it has just finitely many orbits in its action on $n$-tuples for every $n$. A structure $\bA$ is called oligomorphic if its automorphism group is oligomorphic.

Before coming to the definition of weakly oligomorphic structures, we have to recall some model theoretic notions: Let $\Sigma$ be a relational signature, and let $L(\Sigma)$ be the language of first order logics with respect to $\Sigma$. Let $\bA$ be a $\Sigma$-structure. For a formula $\varphi(\bx)$ (where $\bx=(x_1,\dots,x_n)$), we define $\varphi^\bA\subseteq A^n$ as the set of all $n$-tuples $\ba$ over $A$ such that $\bA\models\varphi(\ba)$. More generally, for a set $\Phi$ of formulae from $L(\Sigma)$ with free variables from $\{x_1,\dots,x_n\}$, we define $\Phi^\bA$ as the intersection of all relations $\varphi^\bA$ where $\varphi$ ranges through $\Phi$. We call $\Phi$ a \emph{type}, and $\Phi^\bA$ the relation defined by $\Phi$ in $\bA$.

If $\Phi^\bA\neq\emptyset$, then we say that $\bA$ realizes $\Phi$. We call $\Phi$ positive existential  if it consists just of positive existential formulae. 

For a relation $\varrho\subseteq A^n$ by $\Tp_\bA(\varrho)$ we denote the set of all formulae $\varphi(\bx)$ such that $\varrho\subseteq \varphi^\bA$. This is the type defined by $\varrho$ with respect to $\bA$. Analogously, the positive existential type $\pTp_\bA(\varrho)$ is defined. 
By $\Th(\bA)$ we will denote the full first order theory of $\bA$ while with $\pTh(\bA)$ we will denote the positive existential part of $\Th(\bA)$.

Let us come now to the definition of the structures under consideration in this paper. 
\begin{definition}[\cite{MPPhD,AUpaper}]
	A relational structure $\bA$ is called \emph{weakly oligomorphic} if for every arity there are just finitely many relations that can be defined by positive existential types.
\end{definition}
One can argue that it would be more appropriate to define a structure $\bA$ to be weakly oligomorphic if its endomorphism monoid is \emph{oligomorphic} (i.e.\ there are just finitely many invariant relations of $\End (\bA)$ of any arity). However, there is no need to worry, since, at least for countable structures, these two definitions are equivalent:
\begin{proposition}[{\cite[Thm. 6.3.4]{AUpaper}}, cf.\  {\cite[Prop. 2.2.5.1]{MPPhD}}]\label{prop:weakly:olig}
	A countable structure $\bA$ is weakly oligomorphic if and only if $\End(\bA)$ is oligomorphic.  
\end{proposition}
Clearly, if a structure is oligomorphic, then it is also weakly oligomorphic. Moreover, weakly oligomorphic structures are closed under retracts. Hence, e.g., retracts of $\aleph_0$-categorical structures are weakly oligomorphic.

First structural results for weakly oligomorphic structures were obtained in \cite{MasPec11,AgeHHnew,AUpaper}. 

\section{A motivating result}\label{s2}
In this section we will motivate in greater detail, why we are interested in studying the hom-equivalence classes of weakly oligomorphic structures. The first reason comes from the theory of homogeneous relational structures.  
Let us recall \Fraisse's theorem:
\begin{theorem}[\Fraisse{} (\cite{Fra53})]
  A class $\cC$ of finite relational structures is the age of some
  countable homogeneous relational structure if and only if
  \begin{enumerate}[(i)]
  \item it is closed under isomorphism,
  \item it has only countably many non-isomorphic members,
  \item it has the hereditary property and the amalgamation property.
  \end{enumerate}
  Moreover, any two countable homogeneous relational structures
  with the same age are isomorphic.\qed
\end{theorem}
In \cite{CamNes06}, Peter Cameron and Jaroslav Ne\v set\v ril introduced the notion of homomorphism-homogeneous structures. \emph{A local homomorphism} of a structure $\bA$ is a homomorphism from a finite substructure of $\bA$ to $\bA$. A structure $\bA$ is called \emph{homomorphism-homogeneous} if every local homomorphism of $\bA$ can be extended to an endomorphism of $\bA$.
For homomorphism-homogeneous relational structures a \Fraisse-type characterization was given in \cite{AgeHHnew}:
\begin{theorem}[{\cite[Thm.4.4,Cor.4.6]{AgeHHnew}}]
  A class $\cC$ of finite relational structures is the age of some
  countable homomorphism-homogeneous relational structure if and only if
  \begin{enumerate}[(i)]
  \item it is closed under isomorphism,
  \item it has only countably many non-isomorphic members,
  \item it has the hereditary property and the homo-amalgamation property.
  \end{enumerate}
  Moreover, any two countable homomorphism-homogeneous relational structures
  with the same age are homomorphism-equivalent.\qed
\end{theorem}
If we add to this the easy observation that all homomorphism-homogeneous structures over a finite signature are weakly oligomorphic (cf. \cite[Cor.6.7]{AgeHHnew}), then this gives the first reason to be interested in the hom-equivalence classes of weakly oligomorphic structures.  

Our second motivation comes from the theory of $\aleph_0$-categorical structures. Let us recall the classical result that links $\aleph_0$-categoricity with oligomorphy (cf. \cite[(2.10)]{Cam90}):
\begin{theorem}[Engeler, Ryll-Nardzewski, Svenonius]\label{thm:ryll-nardzewsky}
	Let $\bA$ be a countably infinite relational structure. Then $\Th(\bA)$ is $\aleph_0$-categorical if and only if $\Aut(\bA)$ is oligomorphic.\qed
\end{theorem}
A consequence of this characterization is that two countable oligomorphic structures have the same first order theory if and only if they are isomorphic. For countable weakly oligomorphic relational structures we will show the following related result:
\begin{theorem}\label{woRN}
	Let $\bB$ be a weakly oligomorphic structure, and let $\bA$ be a countable relational structure. Then the following are equivalent:
	\begin{enumerate}
	\item $\bA\to\bB$,
	\item $\pTh(\bA)\subseteq\pTh(\bB)$,
	\item $\Age(\bA)\to\Age(\bB)$,
	\item $\CSP(\bA)\subseteq\CSP(\bB)$.
	\end{enumerate}
\end{theorem}
Before coming to the proof of Theorem~\ref{woRN}, we need to recall a result from \cite{AgeHHnew}, and to prove some additional auxiliary results:
\begin{lemma}[{\cite[Prop.5.2]{AgeHHnew}}]\label{prop:wo-hom}
	Let $\bA$, $\bB$ be a relational structures over the same signature such that $\Age(\bA)\to\Age(\bB)$, and suppose that $\bA$ is countable, and that $\bB$ is weakly oligomorphic. Then $\bA\to\bB$.\qed
\end{lemma}

\begin{lemma}\label{lem_wo_real} 
	Let $\bA$ be a weakly oligomorphic structure over the signature $R$, and let $\Psi$ be a positive existential type in the language of $R$-structures. If every finite subset of $\Psi$ is realized in $\bA$, then $\Psi$ is realized in $\bA$.
\end{lemma}
\begin{proof} 
	Suppose that every finite subset of $\Psi$ is realized in $\bA$, but $\Psi$ is not. Following we will define a  sequence  $(\varphi_i)$ of formulae from $\Psi$, and a sequence $(\bd_i)$ such that $\bA\models\varphi_j(\bd_i)$ for $1\le j \le i$, but $\bA\nvDash \varphi_{i+1}(\bd_i)$.

	Let $\varphi_0\in\Psi$. Then there exists a $\bd_0\in A^m$ such that $\bA\models\varphi_0(\bd_0)$. Suppose that $\varphi_i$, and $\bd_i$ are defined already. By assumption, $\bd_i$ does not  realize $\Psi$. Let $\varphi_{i+1}\in\Psi$ such that  $\bA\nvDash\varphi_{i+1}(\bd_i)$. Again, by assumption, the set $\{\varphi_0,\dots,\varphi_{i+1}\}$ is realized in $\bA$. Let $\bd_{i+1}\in A^m$ a tuple that realizes $\{\varphi_0,\dots,\varphi_{i+1}\}$.

	By construction, the sets $\Psi_i=\{\varphi_1,\dots,\varphi_i\}$ define an infinite decreasing chain of  distinct non-empty relations in $\bA$. However, this gives a contradiction with the assumption that $\bA$ is weakly oligomorphic.

	We conclude that $\Psi$ is realizable.
\end{proof}

\begin{lemma}\label{lem:wo_age} 
	Let $\bA$, and $\bB$ be relational structures over the same  signature.  If $\pTh(\bA)\subseteq \pTh(\bB)$, and if $\bB$ is weakly oligomorphic, then $\Age(\bA)\to\Age(\bB)$.
\end{lemma}
\begin{proof}
	Let $\bC$ be a finite substructure of $\bA$, and let  $C=\{c_1,\dots,c_n\}$ be its carrier. Define  $\bc:=(c_1,\dots,c_n)$. Then every finite subset of  $\pTp_\bA(\bc)$ is realized in $\bB$. Since $\bB$ is weakly oligomorphic, by Lemma~\ref{lem_wo_real}, we get that there is a tuple $\bd\in B^n$ that realizes $\pTp_\bA(\bc)$. Let $D=\{d_1,\dots,d_n\}$, and let $\bD$ be the substructure of $\bB$ induced by $D$. Then the mapping $f:\bC\to\bD$ given by $c_i\mapsto d_i$ is a homomorphism. This shows that $\Age(\bA)\to\Age(\bB)$.
\end{proof}

\begin{proof}[Proof of Theorem~\ref{woRN}]
	($1\Rightarrow 4$) Clear.
	
	($4\Rightarrow 2$) It is well known (and easy to see) that for every positive primitive proposition $\varphi$ there exists a finite relational structure $\bA_\varphi$ such that for any relational structure $\bC$ of the given type, we have $\bC\models\varphi$ if and only if $\bA_\varphi\to\bC$. Thus from $\CSP(\bA)\subseteq\CSP(\bB)$ it follows that the positive primitive theory of $\bA$ is contained in the positive primitive theory of $\bB$. However, this is the case if and only if $\pTh(\bA)\subseteq\pTh(\bB)$.
	
	($2\Rightarrow 3$) This is a direct consequence of Lemma~\ref{lem:wo_age}.
	
	($3\Rightarrow 1$) This follows from Lemma~\ref{prop:wo-hom}.
\end{proof}
Let us recall now a result from \cite{MasPec11}:
\begin{proposition}[{\cite[Th.3.5]{MasPec11}}]
	Let $\bA$ be a countable weakly oligomorphic structure, and let $\bB$ be a countable model of $\Th(\bA)$. Then $\bB$ is weakly oligomorphic, too.\qed
\end{proposition} 
We can combine this with Theorem~\ref{woRN} to obtain:
\begin{corollary}
	Let $T$ be the complete first order theory of a weakly oligomorphic structure. Then all countable models of $T$ are homomorphism-equivalent.\qed
\end{corollary}
A first order theory, for which all countable models are homomorphism-equivalent can rightfully be called \emph{weakly $\aleph_0$-categorical}. Hence we just proved that the first order theory of each  weakly oligomorphic structure is weakly $\aleph_0$-categorical. Models of weakly $\aleph_0$-categorical theory are another motivation to study hom-equivalence classes of weakly oligomorphic structures.

\section{Hom-equivalence classes of weakly oligomorphic structures}\label{s3}

We are going to give structure to  hom-equivalence classes by equipping them with a quasi-order. There are several natural and interesting ways to do so---e.g. we may say that $\bA$ is below $\bB$ if $\bA$ is a retract of $\bB$, or if $\bA$ is a homomorphic image of $\bB$. In this paper we choose the embedding quasi-order (with the hope to come back in the future to the retraction-quasi-order). We define the relation of embeddability on relational structure of a given type. In particular, we write $\bA\injto \bB$ if there exists an embedding from $\bA$ into $\bB$. This relation, clearly defines a quasi-order on the relational structures of a given type. We will restrict this quasi-order to hom-equivalence classes of weakly oligomorphic structures. Abusing the terminology, we will further restrict our attention solely to the finite and countably infinite structures. So for a relational structure, the hom-equivalence class $\HE(\bA)$ is defined to be the class of all countable relational structures $\bB$ for which $\bA\leftrightarrow\bB$.

We are far away from completely understanding the structure embeddability-quasiorder on  hom-equivalence classes of weakly oligomorphic structures. Initially, we are interested in extremal elements---i.e. such elements that embed into all other elements and such elements into which all other elements can be embedded. This promises to be most rewardable as we can expect some exceptional properties of such structures. Of course, being a quasi-order, it can happen that there are several smallest and greatest elements with respect to embeddability, and that not all such elements are of the same beauty. So we aim also to find among all smallest and greatest elements  the most natural, distinguished candidates. 

\subsection{Smallest elements}
For the existence of smallest elements in hom-equivalence classes of weakly oligomorphic structures we only need to collect a couple of results from the literature. Recall that a structure $\bA$ is called a \emph{core} if all its endomorphisms are embeddings. Clearly, if $\bA$ is a smallest element in $\HE(\bT)$, then $\bA$ must be a core. And moreover, every core in $\HE(\bT)$ is a smallest element. Indeed, if $\bB\in\HE(\bT)$ and and if $f:\bA\to\bB$, $g:\bB\to\bA$, then $g\circ f$ is an embedding. Hence $f$ is an embedding, too.

The most simple case is when $\bT$ is finite. In this case up to isomorphism $\bB$ has a unique retract that is a core. This is up to isomorphism the unique smallest element in $\HE(\bT)$.

Let us now treat the case that $\bT$ is homomorphism-homogeneous and weakly oligomorphic. Then we have:
\begin{proposition}[{\cite[Cor.6.10]{AgeHHnew}}]\label{wo-olig}
	Every countable, weakly oligomorphic, homomorphism-homogeneous structure $\bT$ contains, up to isomorphism, a unique homomorphism-equivalent homomorphism-homogeneous core $\bF$. Moreover, $\bF$ is oligomorphic and homogeneous.\qed
\end{proposition}
So in this case we have the existence of a smallest element as well as the existence of a distinguished smallest element (namely, it is the unique one that is homomorphism-homogeneous).

The next case is that $\bT$ is a countably infinite structure with an oligomorphic automorphism group. By the Ryll-Nardzewski theorem, $\bB$ is $\aleph_0$-categorical. Now we can use a result by Bodirsky:
\begin{theorem}[{\cite[Thm.16]{Bod07}}]\label{bod}
	Every $\aleph_0$-categorical relational structure $\bT$  is homomorphism-equivalent to a model-complete core $\bC$, which is unique up to isomorphism, and $\aleph_0$-categorical or finite. For all $k\ge 1$, the orbits of $k$-tuples in $\bC$ are primitive positive definable.\qed
\end{theorem}
So in this case we have the existence of a distinguished smallest element in $\HE(\bT)$---the unique model-complete core.

Finally, let us consider the most general case, that $\bT$ is countably infinite and weakly oligomorphic. In this case we can combine Theorem~\ref{bod} by Bodirsky with the following result:
\begin{proposition}[{\cite[Thm.7.2]{AgeHHnew}}]
	Let $\bT$ be a countable weakly oligomorphic relational structure. Then $\bT$, is homomorphism-equivalent to a finite or $\aleph_0$-categorical structure $\bC$. Moreover, $\bC$ embeds into $\bA$.\qed
\end{proposition}
That is, $\bT$ is homomorphism-equivalent to a model complete core, which is unique up to isomorphism.

\subsection{Greatest elements}\label{s32}
A largest element in the hom-equivalence class of a relational structure is also known under the name of a universal object for this class. It is a long standing open problem, to characterize such classes of countable structures, that have a universal structure. Much work has been done on classes of structures with forbidden substructures. Striking results in this direction are, e.g., \cite{CheSheChi99,HubNes09} (see in particular \cite{HubNes09} for further references).

Let $\bT$ be a relational structure. Then the greatest possible age of a structure on $\HE(\bT)$ is $\CSP(\bT)$. 
If $\bT$ is weakly oligomorphic, and if $\bA$ is any countable structure that is universal for $\overline{\CSP(\bT)}$ then, by Lemma~\ref{prop:wo-hom}, it follows that $\bA$ is a greatest element of $\HE(\bT)$. For finite relational signatures, the existence of such universal structures follows already from the following result by Hubi\v{c}ka and Ne\v{s}et\v{r}il. In this theorem, $\Forb_h(\cF)$ denotes the class of all countable $R$-structures into which no element from $\cF$ maps homomorphically. Moreover, a structure is called connected if its Gaifman-graph is connected. 
\begin{theorem}[{\cite[Thm.2.2]{HubNes09b}}]\label{Hub}
	Let $\cF$ be a countable set of finite connected relational structures over a finite signature $R$. Then $\Forb_h(\cF)$ contains  a universal structure $\bU_\cF$.\qed
\end{theorem}
Indeed,  for a structure $\bT$ over a finite relational signature $R$, it is easy to see that $\overline{\CSP(\bT)}=\Forb_h(\cF)$ where $\cF$ is equal to the class of all finite $R$-structures not in $\CSP(\bT)$. Moreover, we have that whenever a structure in $\cF$ is disconnected, then at least one of its connected components is in $\cF$, too, for otherwise $\Forb_h(\cF)$ would not be closed with respect to direct sums---a contradiction. It follows that if $\cF_c$ denotes the connected structures in $\cF$, then $\Forb_h(\cF)=\Forb_h(\cF_c)$. 
On the other hand, every class of the shape $\Forb_h(\cF)$ for a class of finite connected structures $\cF$, is  of the shape $\overline{\CSP(\bT)}$ for some countable structure $\bT$. Finally, since $R$ is finite, any class of finite $R$ structures contains up to isomorphism just countably many elements. 

Unfortunately, in general we do not know much about the symmetries of the universal structures given by Theorem~\ref{Hub}. However, in the special case that $\cF$ is finite we know of the existence of an $\aleph_0$-categorical universal structure in $\Forb_h(\cF)$ due to a result by  Hubi\v{c}ka and Ne\v{s}et\v{r}il \cite[Thm.1.3]{HubNes09} which in turn generalizes the analogous result about graphs due to Cherlin, Shelah, and Chi (cf. \cite[Thm.4]{CheSheChi99}).

Though, we know now that the hom-equivalence class of any weakly oligomorphic structure has a largest element, in general we do not know of the existence of a ``distinguished'' largest element. Here, by distinguished we, ideally, understand $\aleph_0$-categorical because of the following old result by Saracino:
\begin{theorem}[{\cite[Thms.1,2]{Sar73}}]\label{saracino}
	Let $T$ be an $\aleph_0$-categorical theory with no finite models. Then $T$ has a model-companion $T'$. Moreover, $T'$ is $\aleph_0$-categorical, too.
\end{theorem}
An immediate consequence is, that for every countably infinite, $\aleph_0$-categorical structure $\bU$ there exists a countably infinite $\aleph_0$-categorical structure $\bU'$, such that $\Age(\bU)=\Age(\bU')$, and such that $\bU'$ is model-complete. Moreover $\bU'$ is uniquely determined, up to isomorphism.

Consequently, whenever we have an $\aleph_0$-categorical universal element in a hom-equivalence class, then we also have a distinguished $\aleph_0$-categorical universal element---namely, the model-complete one.

In order to be able to state our contribution to the existence problem of universal structures, we have to introduce the notion of strict amalgamation classes due to Dolinka (cf. \cite[Sec.2.1]{Dol12}). Let $\cC$ be a \Fraisse-class (not necessarily of relational structures). Then we say that $\cC$ has the \emph{strict amalgamation property} if every pair of morphisms in $(\cC,\injto)$ with the same domain has a pushout in $(\overline{\cC},\rightarrow)$. Note that these pushouts will always be amalgams. Thus the strict amalgamation property postulates canonical amalgams. 

Similarly, we say that $\cC$ has the \emph{strict joint embedding property} if every pair of structures from $\cC$ has a coproduct in $(\overline{\cC},\rightarrow)$. If $(\overline\cC,\injto)$ has a finite initial object, then the strict joint embedding property follows from the strict amalgamation property. This is always the case for classes of relational structures.

Finally, a \emph{strict \Fraisse-class} is a \Fraisse-class that enjoys the strict joint embedding property and the strict amalgamation property.

If $\cU$ is a strict \Fraisse-class, then a \Fraisse-class $\cC$ that is a subclass of $\cU$,  will be called \emph{free in $\cC$} if it is closed with respect to finite coproducts and with respect to canonical amalgams in $\cU$.

Note that every free amalgamation class of relational structures over the signature $R$, in our terminology, is a free \Fraisse-class in the class of all finite relational $R$-structures.  Moreover, every free amalgamation class is also a strict \Fraisse-class. However, there are strict \Fraisse-classes that are not free amalgamation classes. The class of finite partial orders is an example.

The following is an instantiation into the model-theoretic world of a more general category-theoretic result, that will be presented in Section~\ref{xs6}:
\begin{theorem}\label{univ-struc}
	Let $\cU$ be a strict \Fraisse-class of relational structures, and let $\cC$ be a \Fraisse-class that is free in $\cU$. Let $\bT\in\overline\cU$.  Then
	\begin{enumerate}
		\item $\overline{\cC}\cap\overline{\CSP(\bT)}$ has a universal element $\bU_{\cC,\bT}$,
		\item if the \Fraisse-limit of $\cC$ and $\bT$ each have an oligomorphic automorphism group (i.e. each is finite or $\aleph_0$-categorical), then  $\overline{\cC}\cap\overline{\CSP(\bT)}$ has a universal element $\bU_{\cC,\bT}$ that is finite or $\aleph_0$-categorical.
	\end{enumerate}
	If $\bT\in\overline{\cC}$, then $\bU_{\cC,\bT}$ can be chosen as a co-retract of $\bT$.
\end{theorem}
We will  postpone the proof of Theorem~\ref{univ-struc} to Section~\ref{xs3}, as we need to to construct some tools first. It is our goal to construct $\bU_{\cC,\bT}$ as the reduct of a \Fraisse-limit of a class of expanded structures (a usual method). However, our way of expanding structures will leave the domain of model theory. In the next section we will elaborate on this  by introducing and studying constraint relational structures.

Before starting, let us give a direct consequence of Theorem~\ref{univ-struc} that is relevant to hom-equivalence classes.
\begin{corollary}
	Let $R$ be a countable relational signature, and let $\bT$ be a countable $R$-structure. Then $\HE(\bT)$ has a largest element. Moreover, if $R$ is finite and $\bT$ is weakly oligomorphic, then $\HE(\bT)$ has, up to isomorphism, a unique $\aleph_0$-categorical, model complete largest element.
\end{corollary}
\begin{proof}
	Note that the class of all finite $R$-structures is a free amalgamation class, and hence a strict \Fraisse-class. Denote this class by $\cU$. Let $\cC:=\cU$. Then, by Theorem~\ref{univ-struc}, there exists a universal element $\bU_\bT$ in $\overline{\CSP(\bT)}$. Moreover, $\bT$ is a retract of $\bU_\bT$. In particular, $\bU_\bT$ is in $\HE(\bT)$. 
	
	If $R$ is finite, then $\cC$ has an $\aleph_0$-categorical \Fraisse-limit. By Proposition~\ref{wo-olig}, it follows that $\bT$ is homomorphism-equivalent to a finite or $\aleph_0$-categorical structure. So we can without loss of generality assume that the automorphism group of $\bT$ is oligomorphic. However, then Theorem~\ref{univ-struc} gives us additionally, that $\bU_\bT$ can be chosen to be finite or $\aleph_0$-categorical. Since $\CSP(\bT)$ contains structures of arbitrary size, we conclude that $\bU_\bT$ is $\aleph_0$-categorical. Now the existence and uniqueness of an $\aleph_0$-categorical, model-complete universal structure in $\HE(\bT)$ follows from Theorem~\ref{saracino}.
\end{proof}

\section{Constraint structures}\label{s4}
Throughout this section we fix a relational signature $R$. Also we fix a universe $\cU$ of countable $R$-structures which we require to be a strict \Fraisse-class. 

Let $\bT\in\overline{\cU}$, and let $\cC$ be a \Fraisse-class that is free in $\cU$.  A  \emph{$\bT$-colored  structure in $\overline{\cC}$} is a pair $(\bA,a)$ such that $\bA\in\overline{\cC}$ and  $a:\bA\to\bT$ is a homomorphism. If $\cC$ is known from the context, we will usually leave away  ``in $\overline{\cC}$''.

A (strong) homomorphism $f$ between $\bT$-colored structures $(\bA,a)$ and $(\bB,b)$ is a homomorphism from $\bA$ to $\bB$ such that $b\circ f=a$. A homomorphism between $\bT$-colored structures is called \emph{embedding} if it is an embedding between the corresponding $R$-structures. 

With $\Col_\cC(\bT)$ we will denote the class of all countable  $\bT$-colored structures in $\overline{\cC}$. As usual, by $(\Col_\cC(\bT),\to)$ and $(\Col_\cC(\bT),\injto)$ we denote the corresponding categories with homomorphisms and embeddings, respectively. Endomorphisms are homomorphisms of a $\bT$-colored structure to itself and automorphisms are bijective endomorphisms whose inverse is an endomorphism, too. The group of automorphisms of $(\bA,a)$ will be denoted by $\sAut(\bA,a)$ (we use the notion $\sAut$, standing for ``strong'' automorphisms, instead of $\Aut$, because later on we will consider also weak automorphisms of $\bT$-colored structures).

We have the following \Fraisse-type result for $\bT$-colored structures: 
\begin{theorem}\label{hom-constraints}
	With the notions from above,  there exists a countable  $(\bU,u)\in\Col_\cC(\bT)$ such that
	\begin{enumerate}
		\item for every countable $\bT$-colored structure $(\bA,a)\in\Col_\cC(\bT)$ there exists an embedding $\iota:(\bA,a)\injto(\bU,u)$  (universality),
		\item for every finite $\bT$-colored structure $(\bA,a)\in\Col_\cC(\bT)$, and for all embeddings $\iota_1,\iota_2:(\bA,a)\injto(\bU,u)$ there exists an automorphism $f$ of $(\bU,u)$ such that $f\circ\iota_1=\iota_2$ (homogeneity).
	\end{enumerate}
	Moreover, all countable universal homogeneous $\bT$-colored structures are mutually isomorphic. 
\end{theorem}
For proving this theorem we are standing at a cross-road and see three ways to proceed. Either we translate the $\bT$-colored structures to models and use \Fraisse's theorem, or we prove yet another \Fraisse-type theorem especially for $\bT$-colored structures, or we use an already available \Fraisse-type theorem for categories, such as the one due to Droste and G\"obel (cf. \cite{DroGoe92}). All three possibilities require considerable technical work. We decided for the third way, as it seems to us the cleanest and it promises to be the most useful in future research. The necessary techniques are quite independent of the rest of the paper. Therefore, in order not to disturb the flow of presentation, we postpone the category theoretical part to  Section~\ref{s6}, and the actual proof of Theorem~\ref{hom-constraints} to Section~\ref{xs4}.

Let us have a look onto the symmetries of universal homogeneous $\bT$-colored structures:
\begin{proposition}\label{hom-const-saut}
	Let $(\bU,u)$ be a universal homogeneous $\bT$-colored structure in $\overline\cC$. If the \Fraisse-limit of $\cC$ is finite or $\aleph_0$-categorical, and if $\bT$ is finite, then $\sAut(\bU,u)$ is oligomorphic. 
\end{proposition}
\begin{proof}	
	Let $\ba=(a_1,\dots,a_n)$, and $\bb=(b_1,\dots,b_n)$ be tuples of elements from $U$, such that the mapping $a_i\mapsto b_i$ is a local isomorphism of $R$-structures, and suppose that $u(\ba)=u(\bb)$. Let $\bA$ be the substructure of $\bU$ that is induced by $\{a_1,\dots,a_n\}$, $a:\bA\to\bT$ be the restriction of $u$ to $\bA$. Let  $\iota_1:(\bA,a)\injto (\bU,u)$  be the identical embedding, and  define $\iota_2 : (\bA,a)\injto (\bU,u)$ by $a_i\mapsto b_i$. Then, by the assumptions on $\ba$ and $\bb$,  $\iota_2$ is an embedding, and by the homogeneity of $(\bU,u)$, there exists an automorphism $f$ of $(\bU,u)$ such that $f\circ\iota_1=\iota_2$. In particular, $f(\ba)=\bb$. Since there are only finitely many isomorphism types of $n$-tuples in $\bU$, and since there are only finitely many $n$-tuples in $\bT$, it follows that  $\sAut(\bU,u)$ has only finitely many $n$-orbits. Hence it is oligomorphic.
\end{proof}

So far, what we know about the symmetries of $\bU$ is not enough to prove Theorem~\ref{univ-struc}. For an $\aleph_0$-categorical structure $\bT$ the strong automorphism group of $(\bU,u)$ does not yield enough information about $\bU$. For this reason we will add morphisms to the category of $\bT$-colored structures: 

A \emph{weak homomorphism} from $(\bA,a)$ to $(\bB,b)$ is a pair $(f,g)$ such that $f:\bA\to\bB$ is a homomorphism, $g$ is  is an automorphism of $\bT$, such that $b\circ f = g \circ a$. Weak embeddings are weak homomorphisms whose first component is an embedding. Accordingly, the group of weak automorphisms of $(\bA,a)$ is defined like
\[\wAut(\bA,a):=\{(f,g)\mid f\in\Aut(\bA),\,a\circ f=g\circ a\}.\]
We define $(f_1,g_1)\circ(f_2,g_2):=(f_1\circ f_2,g_1\circ g_2)$.
With $(\Col_\cC(\bT),\to_w)$ and $(\Col_\cC(\bT),\injto_w)$ we will denote the categories of $\bT$-colored structures in $\overline{\cC}$ with weak homomorphisms and weak embeddings, respectively.

The group of strong automorphisms embeds naturally into the group of strong automorphisms through $f\mapsto (f,1_\bT)$. 

Clearly, the two projections $\pi_1:(f,g)\mapsto f$, $\pi_2:(f,g)\mapsto g$ are group homomorphisms from $\wAut(\bA,a)$ to $\Aut(\bA)$ and $\Aut(\bT)$, respectively. The image of $\pi_1$ will be called the \emph{color automorphism group} of $(\bA,a)$, and it will be denoted by $\cAut(\bA,a)$.  The kernel of $\pi_2$ is isomorphic to $\sAut(\bA,a)$.

Let us call a $\bT$-colored structure $(\bU,u)\in\Col_\cC(\bT)$ \emph{w-homogeneous} if for every finite $\bT$-colored structure $(\bA,a)\in\Col_\cC(\bT)$, and for all weak embeddings $(f_1,g_2),(f_2,g_2):(\bA,a)\injto(\bU,u)$ there exists a weak automorphism $(f,g)$ of $(\bU,u)$ such that $(f,g)\circ(f_1,g_1)=(f_2,g_2)$.

Then we have the following \Fraisse-type result whose proof will be postponed till Section~\ref{xs4}, when all technical prerequisites are provided: 
\begin{theorem}\label{hom-constraints2}
	Let $(\bU,u)$ be a  universal homogeneous $\bT$-colored structure in $\overline{\cC}$. Then $(\bU,u)$ is w-homogeneous. 
	Moreover, all countable  $\bT$-colored structures in $\overline{\cC}$ that are universal and $w$-homogeneous are mutually isomorphic. 
\end{theorem}

Our next observation is:
\begin{lemma}
	Let $(\bA,a)$ be a $\bT$-colored structure in $\overline\cC$, such that $a$ is an epimorphism. Then $\wAut(\bA,a)\cong\cAut(\bA,a)$. 
\end{lemma}
\begin{proof}
	It is enough to show that the projection $\pi_1$ is injective. Let $(f,g_1)$, $(f,g_2)$ be weak automorphisms of $(\bA,a)$. Then, by definition $a\circ f=g_1\circ a= g_2\circ a$. Since $a$ is an epimorphism, it follows that $g_1=g_2$. 
\end{proof}

Let us now reexamine, when universal, homogeneous $\bT$-colored structures give rise to $\aleph_0$-categorical structures:
\begin{proposition}\label{hom-const-waut}
	Let $(\bU,u)$ be a universal, homogeneous $\bT$-colored structure in $\overline\cC$. If the \Fraisse-limit of $\cC$ is finite or $\aleph_0$-categorical, and if $\bT$ has an oligomorphic automorphism group, then $\cAut(\bU,u)$ is oligomorphic. 
\end{proposition}
\begin{proof}	
	Let $\ba=(a_1,\dots,a_n)$, and $\bb=(b_1,\dots,b_n)$ be tuples of elements from $U$, such that the mapping $a_i\mapsto b_i$ is a local isomorphism of $R$-structures, and suppose that $u(\ba)$ and $u(\bb)$ are in the same $n$-orbit of $\Aut(\bT)$. 
	
	Let $\bA$ be the substructure of $\bU$ that is induced by the entries of $\ba$. Let $\iota$ be the identical embedding of $\bA$ into $\bU$, and let $f:\bA\to\bU$ given by $f(a_i)=b_i$. Clearly, both $\iota$ and $f$ are homomorphisms (of $R$-structures). Let $a:=\iota\circ h$. Then $(\bA,\iota\circ u)\in\Col_\cC(\bT)$. Clearly, $(\iota,1_\bC):(\bA,\iota\circ u)\injto(\bU,u)$. Since $u(\ba)$ and $u(\bb)$ are in the same $n$-orbit of $\Aut(\bT)$, there exists a $g\in \Aut(\bT)$ that maps $u(\ba)$ to $u(\bb)$. It is now easy to see that $(f,g):(\bA,\iota\circ u)\injto(\bU,u)$ is a weak embedding. Hence, by w-homogeneity of $(\bU,u)$, there exists a weak automorphism $(v,w)$ of $(\bU,u)$ such that $(v,w)\circ(\iota,1_\bC)=(f,g)$. In particular, $v$ is a color automorphism of $(\bU,u)$ that maps $\ba$ to $\bb$. 
		
	Since there are only finitely many isomorphism types of $n$-tuples in $\bU$, and since $\Aut(\bT)$ has only  finitely many $n$-orbits, it follows that $\cAut(\bU, u)$ has only finitely many $n$-orbits. Hence it is oligomorphic. 
\end{proof}

\begin{proposition}\label{retract}
	Suppose that $\bT\in\overline{\cC}$.
	Let $(\bU,u)$ be a universal homogeneous $\bT$-colored structure in $\overline\cC$. Then $u:\bU\to\bT$ is a retraction.
\end{proposition}
\begin{proof}
	Since $\bT\in\overline{\cC}$, we have that $(\bT,1_\bT)\in\Col_\cC(\bT)$. Hence, there exists an embedding $\iota:(\bT,1_\bT)\injto(\bU,u)$. That is, we have $u\circ\iota=1_\bT$. Thus, $\iota$ is a right-inverse of $u$, and $u$ is a retraction. 
\end{proof}
An immediate consequence is:
\begin{corollary}
	Suppose that $\bT\in\overline{\cC}$.
	Let $(\bU,u)$ be a universal homogeneous $\bT$-colored structure in $\overline{\cC}$. Then $\cAut(\bU,u)\cong\wAut(\bU,u)$.\qed
\end{corollary}

In the following it will be our goal to represent the color automorphism group of a universal homogeneous $\bT$-colored structure as the automorphism group of a relational structure. 
First, let us expand the signature $R$ of $\bT$ by another binary relational symbol $\kappa$ and denote the new signature by $\tilde{R}$. For a $\bT$-colored structure $(\bA,a)$ define $\widetilde{\bA}$ to be the $\tilde{R}$ structure whose $R$-reduct is equal to $\bA$ and in which the relational symbol $\kappa$ is interpreted as the kernel of $a$.
\begin{proposition}\label{locally_closed}
	Suppose that $\bT\in\overline{\cC}$.
	Let $(\bU,u)$ be a universal homogeneous $\bT$-colored structure in $\overline{\cC}$. Then $\cAut(\bU,u)=\Aut(\widetilde{\bU})$ (in particular, $\cAut(\bU,u)$ is closed in $S_\infty$).
\end{proposition}
\begin{proof}
	Let $\kappa$ be the kernel of $u$. For every $\varrho\in R$ we introduce a new symbol $\hat{\varrho}$ and define $\hat{\varrho}_\bU:=\{\ba\mid u(\ba)\in\varrho_\bT\}$. Let $\widehat{\bU}$ be the structure obtained from $\bU$ by expansion through $\kappa$ and through all $\hat{\varrho}_\bU$. We claim that $\Aut(\widehat{\bU})=\cAut(\bU,u)$.   
	
	Let $f\in\cAut(\bU,u)$. Then there exists a $g\in\Aut(\bT)$ such that $g\circ u= u\circ f$. If $(a,b)\in\kappa$ then $(a,b)\in\ker(g\circ u)$ hence $(a,b)\in\ker(u\circ f)$ whence $(f(a),f(b))\in\kappa$. Let $\ba\in\hat{\varrho}_\bU$. Then $g(u(\ba))\in \varrho_\bT$. Hence $u(f(\ba))\in\varrho_\bT$, whence $f(\ba)\in\hat{\varrho}_\bA$. We conclude that $f\in\Aut(\widehat{\bU})$. 
	
	Let now $f\in\Aut(\widehat{\bU})$. Let us define a binary relation $\gamma$ on $U$ by $\gamma:=\{(u(a),u(f(a)))\mid a\in U\}$. We claim that $\gamma$ is the graph of an automorphism of $\bT$. 
	If $(x,y_1),(x,y_2)\in\gamma$ then there exists $a_1,a_2\in U$ such that $u(a_1)=u(a_2)=x$ and $u(f(a_1))=y_1$, $u(f(a_2))=y_2$. But then $(a_1,a_2)\in \kappa$ and since $f$ preserves $\kappa$, we get $y_1=y_2$. Since $u$ is surjective (recall that it is a retraction of $\bU$ onto $\bT$), it follows that $\gamma$ is the graph of a function. Suppose now that $(x_1,y),(x_2,y)\in\gamma$. Then there are $a_1,a_2\in U$ such that $u(a_1)=x_1$, $u(a_2)=x_2$, $u(f(a_1))=u(f(a_2))=y$. Hence $(f(a_1),f(a_2))\in\kappa$. Since $f^{-1}$ is an automorphism of $\widehat{\bU}$, we have  $(a_1,a_2)\in\kappa$ whence $x_1=x_2$. It follows that $\gamma$ is the graph of an injective function. Since $f$ is a bijection, it follows that $\gamma$ is the graph of a bijective function. Let $g$ be this function. Clearly, by construction we have $g\circ u=u\circ f$.
	
	It remains to show that $g$ is an automorphism of $\bT$. Let $\varrho\in R^{(n)}$ and $\bc\in\varrho_\bT$. Let $\ba\in U^n$ such that $u(\ba)=\bc$. Then $\ba\in\hat{\varrho}_\bU$. Hence $f(\ba)\in\hat{\varrho}_\bU$ and $u(f(\ba))\in\varrho_\bT$. By construction of $g$ we have that $u(f(\ba))=g(\bc)$. Hence $g\in\Aut(\bT)$. We conclude that $f\in\cAut(\bU,u)$.
	
	Finally, we have to show that the $\tilde{R}$-reduct $\widetilde{\bU}$ of $\widehat{\bU}$ has the same automorphism group like $\widehat{\bU}$. It is enough to show that the relations $\hat{\varrho}_\bU$ are first order definable in $\widetilde{\bU}$. For $\varrho\in R^{(n)}$ consider the formula
	\[
		\varphi_\varrho(x_1,\dots,x_n) \equiv \exists x_{n+1}\dots\exists x_{2n} \bigwedge_{i=1}^n \kappa(x_i,x_{i+n}) \land \varrho(x_{n+1},\dots x_{2n}).
	\]
	Let $\varphi_{\varrho,\widetilde{\bU}}$ be the interpretation of $\varphi_\varrho$ in $\widetilde{\bU}$. We claim that $\varphi_{\varrho,\widetilde{\bU}}=\hat\varrho_\bU$. Clearly, $\varphi_{\varrho,\widetilde{\bU}}\subseteq\hat\varrho_\bU$. For the other inclusion we use that $u$ is a retraction onto $\bT$. Let $\iota:\bT\injto\bU$ be a co-retraction of $u$. Let $\ba\in \hat{\varrho}_\bU$. Let $\bb:=\iota(u(\ba))$. Then $\bb\in\varrho_\bU$, and $u(\bb)=u(\ba)$. Hence for all $1\le i\le n$ we have that $(a_i,b_i)\in\kappa_\bU$. Hence $\ba\in \varphi_{\varrho,\widetilde{\bU}}$.  
\end{proof}

\section{A class of $\bT$-colored structures with the small index property}\label{s5}

Following \cite{DolMas12,KecRos07}, we say that a \Fraisse-class $\cC$ has the \emph{Hrushovski property} if for every $\bA\in\cC$ there exists a $\bB\in\cC$ such that $\bA\le\bB$ and such that every isomorphism between substructures of $\bA$ extends to an automorphism of $\bB$. 
We say that a homogeneous structure $\bF$ has the Hrushovski property  if $\Age(\bF)$ does.

As the name suggests, the Hrushovski property was first proved for the class of finite simple graphs by Hrushovski in \cite{Hru92}. 

Hodges, Hodkinson, Lascar, and Shelah showed how the Hrushovski property is related to the small index property: Let $G$ be a permutation group acting on a countable set. Then $G$ is said to have the \emph{small index property} if every subgroup of index less than $2^{\aleph_0}$ contains the stabilizer of a finite tuple. A structure $\bA$ is said to have the small index property if its automorphism group does. Herwig, in \cite{Her98}, generalizing Hrushovski's ideas,  gave a sufficient condition for relational structures over finite signatures to have the Hrushovski property and the small index property. 
Let us now reproduce some important notions from \cite{Her98}.  

Let $R$ be a relational signature. A finite $R$-structure $\bA$ is called a \emph{link-structure}, if either $|A|=1$ or there exist $a_1,\dots,a_n\in A$ such that $A=\{a_1,\dots,a_n\}$ and for some $\varrho\in R^{(n)}$ we have $(a_1,\dots,a_n)\in\varrho_{\bA}$ (note that the $a_i$ do not need to be distinct). 

If $\cL$ is a set of link-structures, then we say that a structure $\bA$ has \emph{link type} $\cL$ if every substructure of $\bA$ that is a link structure, is isomorphic to some structure from $\cL$. 

Let $\cF$ be a set of finite $R$-structures. Then an $R$-structure $\bA$ is called \emph{$\cF$-free}, if no member of $\cF$ homomorphically maps to $\bA$. 

If $\cL$  is a set of link-structures and $\cF$ is a set of finite $R$-structures, then by $\KLF$ denotes the class of all finite $\cF$-free $R$-structures of link-type $\cL$. 

A finite $R$ structure $\bA$ is called \emph{packed} if any two distinct elements of $A$ lie in a tuple of some basic relation of $\bA$ (in other words, the Gaifman-graph of $\bA$ is the complete graph). 

Now we are ready to formulate Herwig's criterion:
\begin{theorem}[{\cite[Thm.15]{Her98}}]\label{herwig}
Let $R$ be a finite relational signature. Let $\cF$ be a set of finite packed $R$-structures. Let $\cL$ be a set of link-structures. Then $\KLF$ is a free amalgamation class that has the Hrushovski property. Moreover, the automorphism group of the \Fraisse-limit of $\KLF$  has ample generics and the small index property.\qed
\end{theorem}
We are not going to define the notion of ample generics but only remark that this property (originating from ideas from \cite{Las91,Tru92}, cf. also \cite{Tru07}) is central to showing the small index property of a permutation group.

Note that the classes of the shape $\cK_{\emptyset\cF}$ are precisely the so called monotone free amalgamation classes (cf. \cite[Lem.14]{Her98}). Another way to put this is, that a CSP has the amalgamation property if and only if it can be represented like $\cK_{\emptyset\cF}$ for some set $\cF$ of finite packed structures.

We are not happy about the requirement that the relational signature has to be finite. In order to overcome this problem, we have to restrict the possible link structures. Let $R$ be an arbitrary relational signature. A finite $R$-structure $\bA$ is called \emph{sparse} if all but finitely many relational symbols from $R$ are interpreted by $\emptyset$ in $\bA$. An arbitrary $R$-structure is called sparse if all its finite substructures are sparse. For an $R$-structure $\bA$, by $R(\bA)$ we will denote the set of all relational symbols in $R$ that have a non-empty interpretation in $\bA$. 

Now we can formulate our criterion:
\begin{proposition}\label{extherwig}
	Let $R$ be a relational signature. Let $\cF$ be a set of finite packed $R$-structures. Let $\cL$ be a set of sparse link-structures. Then $\KLF$ has the Hrushovski property. If $\cL$ is countable, then $\KLF$ is a free amalgamation class whose \Fraisse-limit has the small index property. 
\end{proposition}
\begin{proof}
	Let $\bA\in\KLF$, and let $p_1,\dots,p_n$ be  isomorphisms between substructures of $\bA$. Let $R':=R(\bA)$. Since $\bA$ is sparse, it follows that $R'$ is finite. Let $\cL'$ be the subclass of of all link-structures $\bL\in\cL$ for which $R(\bL)\subseteq R'$. Similarly, let $\cF'$ be the subclass of all structure $\bK\in\cF$ for which $R(\bK)\subseteq R'$. Then $\cK_{\cL'\cF'}\subseteq \KLF$ and $\bA\in\cK_{\cL'\cF'}$. By Theorem~\ref{herwig}, it follows that $\cK_{\cL'\cF'}$ has the Hrushovski property. Hence, there is a finite superstructure $\bB$ of $\bA$ in $\cK_{\cL'\cF'}$ such that $p_1,\dots,p_n$ extend to automorphisms of $\bB$. But we have that $\bB$ is also an element of $\KLF$. Hence $\KLF$ has the Hrushovski property, too.
	
	Clearly, $\KLF$ has the hereditary property and it has free amalgams. However, in order to form a \Fraisse-limit, we need in addition, that it contains up to isomorphism just countably many structures. However, this is assured by the requirement, that $\cL$ should be countable. In this case, from the Hrushovski property of $\KLF$ and from the fact that $\KLF$ is a free amalgamation class, it follows that the automorphism group of the \Fraisse-limit of $\KLF$ has ample generics and the small index property (cf. \cite[Prop.8,Thm.11]{Her98}).
\end{proof}

The Hrushovski property can be straight forwardly defined also for \Fraisse-classes of $\bT$-colored structures with respect to strong embeddings: We say that $\Col_\cC(\bT)$  the \emph{Hrushovski property} if for every finite $(\bA,a)\in\Col_\cC(\bT)$ there exists a finite $(\bB,b)\in\Col_\cC(\bT)$ such that $(\bA,a)\le(\bB,b)$ and such that every strong isomorphism between substructures of $(\bA,a)$ extends to a strong automorphism of $(\bB,b)$. Our main result in this section will be:
\begin{theorem}\label{SIP}
	Let $R$ be a relational signature. Let $\cF$ be a set of finite packed $R$-structures and let $\cL$ be a countable set of sparse link-structures. Let  $\bT$  be any countable $R$-structure. Then $\Col_{\KLF}(\bT)$ has the Hrushovski property. If $(\bU,u)$ is a universal homogeneous object in $\Col_\cC(\bT)$, then $\sAut(\bU,u)$ has the small index property.
\end{theorem}
\begin{proof}
The first step will be to encode $\bT$-colored structures  as relational structures over an extended signature $\hat{R}$:  For every $t\in T$ we add a new unary relation-symbol $M_t$ to $R$, obtaining a new signature $\hat{R}$. Now, to every $\bT$-colored structure $(\bA,a)$ we associate an $\hat{R}$-structure $S(\bA,a)$ by setting $M_{t,S(\bA,a)}:=a^{-1}(t)$. Clearly this process is one-to-one. That is, we can reconstruct $(\bA,a)$ from $S(\bA,a)$. Moreover, $f:(\bA,a)\injto(\bB,b)$ if and only if $f:S(\bA,a)\injto S(\bB,b)$.

Of course, not all $\hat{R}$-structures can be obtained in this way. Let $\cL':=\{(\bL,l)\mid \bL\in\cL,\,l:\bL\to\bT\}$ and let  $\widehat\cL:=\{S(\bL,l)\mid (\bL,l)\in\cL'\}$.  
We can consider $\cF$ as a set of $\hat{R}$ structures in a natural way, by interpreting all additional relational symbols with the empty set.
We will show now that $S$ induces a concrete isomorphism between $(\Col_\KLF(\bT),\injto)$ and $(\overline{\cK_{\widehat\cL\cF}},\injto)$. Let $(\bA,a)\in\Col_\KLF(\bT)$. Then every link-structure of $\bA$ induces a $\bT$-colored substructure that is in $\cL'$. Hence it induces a structure from $\widehat\cL$ in $S(\bA,a)$. Moreover, if $\bF\in\cF$, and if $h:\bF\to S(\bA,a)$, then $h:\bF\to \bA$---a contradiction. Hence $S(\bA,a)\in  \overline{\cK_{\widehat\cL\cF}}$. It remains to show that $S$ is surjective. Let $\widehat\bB\in\overline{\cK_{\widehat\cL\cF}}$, and let $\bB$ be its $R$-reduct. Clearly, $\bB\in\overline{\KLF}$. Let $u\in \widehat\bB$. Then $u$ induces a link-structure from $\widehat\cL$. Hence there exists a unique $t_u\in T$ such that $u\in M_{t_u,\widehat\bB}$. Define $b:B\to T$ by $u\mapsto t_u$. We claim that $b:\bB\to\bT$. Suppose this was not the case. Then there exists a relational symbol $\varrho\in R^{(n)}$ and a tuple $\bu\in B^n$ such that $\bu\in\varrho_\bB$ but $b(\bu)\notin\varrho_\bT$. However, then the elements of $\bu$ induce a link-structure in $\widehat{\bB}$ that is not in $\widehat\cL$---contradiction. Thus $(\bB,b)$ is a $\bT$-colored structure in $\overline{\KLF}$. By construction we have $S(\bB,b)=\widehat{\bB}$. Moreover, all elements of $\widehat\cL$ are sparse. Also, since $\bT$ is countable, so is $\widehat\cL$. 

From Proposition~\ref{extherwig}, it follows that $\cK_{\widehat\cL\cF}$ has the Hrushovski property. Hence $\Col_\KLF(\bT)$ has the Hrushovski property. Moreover, the \Fraisse-limit of $\cK_{\widehat\cL\cF}$ has the small index property. Let $(\bU,u)$ be a universal homogeneous $\bT$-colored structure in $\overline\KLF$. Then $S(\bU,u)$ is a universal homogeneous structure in $\overline{\cK_{\widehat\cL\cF}}$. By the construction of $S$ we have $\sAut(\bU,u)=\Aut(S(\bU,u))$. Hence $\sAut(\bU,u)$ has the small index property.   
\end{proof}

\begin{remark}
	If we assume that $\bT$ is an element of $\KLF$, then we can make much stronger statements. From one hand we observe that in this case  there exists a finite tuple $\ba$ over $U$, such that the stabilizer of $\ba$ in $\cAut(\bU,u)$ is contained in $\sAut(\bU,u)$. Hence $\cAut(\bU,u)$ has the small index property, too.

	On the other hand, if $\bT\in\KLF$, then $\sAut(\bU,u)$ is oligomorphic. Moreover, as it was shown above, it is  closed in $S_\infty$ and has ample generics (we refer to \cite{KecRos07} for a definition of this notion). From this it follows that $\sAut(\bU,u)$ is not equal to the union of any $\omega$-chain of proper subgroups (i.e., it has uncountable cofinality, cf. the remark after Theorem~6.12 of \cite{KecRos07}). Moreover,  from \cite[Thm.6.19]{KecRos07} it follows that in this case $\sAut(\bU,u)$  is $21$-Bergman. That is, whenever $W_0\subseteq W_1 \subseteq W_2\subseteq\dots\subseteq \sAut(\bU,u)$ is an exhaustive sequence of of subsets of $\sAut(\bU,u)$, then there is an $n$, such that $W_n^{21}=\sAut(\bU,u)$. In particular, $\sAut(\bU,u)$ has the Bergman-property (cf. \cite{Ber06,DroGoe05}).
\end{remark}

\section{Universal homogeneous objects in comma categories}\label{s6}
In this Section we will provide the main tool needed for proving Theorems~\ref{hom-constraints} and \ref{hom-constraints2}. The exposition is more or less independent from the rest of the paper. The main result will be a theorem that gives sufficient conditions that a comma-category contains a universal homogeneous object. This result relies on an earlier categorical version of \Fraisse's theorem due to Droste and G\"obel  \cite{DroGoe92}.

For the convenience of the reader, this part is kept relatively self-contained. For basic notions from category theory we refer to \cite{HCA1}.

\subsection{Comma-categories}
Recall the definition of the comma categories:
\begin{definition}
   Let  $\fA$,$\fB$,$\fC$ be categories, let $F:\fA\to\fC$, $G:\fB\to\fC$ be functors. The arrow category $(F \downarrow G)$ has as objects triples $(A,f,B)$ where $A\in\fA$, $B\in\fB$, $f:FA\to GB$. The morphisms from $(A,f,B)$ to $(A',f',B')$ are pairs $(a,b)$ and such that $a:A\to A'$, $b:B\to B'$ such that the following diagram commutes:
\[
	\begin{psmatrix}
 		FA & GB\\
		FA' & GB'
		\ncline{->}{1,1}{1,2}^{f}
		\ncline{->}{2,1}{2,2}^{f'}
		\ncline{->}{1,1}{2,1}<{Fa}
		\ncline{->}{1,2}{2,2}<{Gb}
	\end{psmatrix}
\]
\end{definition}
There are two projection functors $U:(F \downarrow G)\to \fA$ and $V:(F \downarrow G)\to \fB$ defined by $U:(A,f,B)\mapsto A, (a,b)\mapsto a$ and $V:(A,f,B)\mapsto B, (a,b)\mapsto b$. Moreover there is a canonical natural transformation $\alpha: F\circ U\to G\circ V$ defined by $\alpha_{(A,f,B)}=f$ (cf. \cite[Prop.1.6.2]{HCA1}). Finally, the comma-category has the following universal property:
\begin{proposition}[{\cite[Prop.1.6.3]{HCA1}}]
	With the notions from above, let $\fD$ be another category, let $U':\fD\to\fA$, $V':\fD\to\fB$ be functors, and let $\alpha':FU'\Rightarrow GV'$ be a natural transformation. Then there exists a unique functor $W:\fD\to(F \downarrow G)$ such that $UW=U'$, $VW=V'$, $\alpha*1_W=\alpha'$ (where $\alpha*1_W$ denotes the horizontal composition of $\alpha$ and $1_W$, i.e. $\alpha'_D=\alpha_{WD}$ for all $D\in\fD$).\qed
\end{proposition}

The following lemmata show how arrow categories behave with respect to (weak) colimits. Nothing here is really new. But since it is not easy to find these facts in literature, we give them here with proof. 
\begin{lemma}\label{weak}
	With the notions from above, let $\fD$ be a small category, and let $H:\fD\to (F \downarrow G)$. Suppose that
	\begin{enumerate}
		\item $U\circ H$ has a weak colimit $(L,(p_D)_{D\in\fD})$,
		\item $V\circ H$ has a compatible cocone $(M,(q_D)_{D\in\fD})$,
		\item $(FL,(Fp_D)_{D\in\fD})$ is a weak colimit of $FUH$.
	\end{enumerate}
	Then there is a morphism $h: FL\to GM$ such that $((L,h,M), (p_D,q_D)_{D\in\fD})$ is a compatible cocone for $H$. In case that $(FL,(Fp_D)_{D\in\fD})$ is a colimit of $FUH$, $h$ is unique.
\end{lemma}
\begin{proof}
	Recall $\alpha:FU\Rightarrow GV$ with $\alpha_{(A,f,B)}=f$  is a natural transformation. Consider $\alpha':=\alpha * 1_H:FUH\Rightarrow GVH$ given by $\alpha'_D=\alpha_{HD}$. Then $(GM,(Gq_D\circ\alpha'_D)_{D\in \fD})$ is a compatible cocone for $FUH$. Hence there is a morphism  $h: FL\to GM$ such that for every $D\in \fD$ we have $Gq_D\circ\alpha'_D=h\circ Fp_D$. So indeed, $(p_D,q_D):HD\to (L,h,M)$. In case that $(FL,(Fp_D)_{D\in\fD})$ is a colimit, there is a unique such $h$ such that  $(p_D,q_D):HD\to(L,h,M)$ is a morphism. 
	
	Let $d: D\to D'$ be a morphism of $\fD$. Then $HD=(UHD, \alpha'_D, VHD)$ and $HD'=(UHD',\alpha'_{D'},VHD')$. Moreover, $Hd=(UHd,VHd)$. So we have that the following diagram commutes:
	\[
		\begin{psmatrix}[rowsep=1cm]
			&[name=FUHD]FUHD & [name=GVHD]GVHD\\
			[name=FL]FL & & & [name=GM]GM\\
			&[name=FUHD']FUHD' & [name=GVHD']GVHD'\\[-1cm]
			\ncline{->}{FUHD}{GVHD}^{\alpha'_D}
			\ncline{->}{FUHD'}{GVHD'}^{\alpha'_{D'}}
			\ncline{->}{FUHD}{FUHD'}<{FUHd}
			\ncline{->}{GVHD}{GVHD'}<{GVHd}
			\ncline{->}{FUHD}{FL}^{Fp_D}
			\ncline{->}{FUHD'}{FL}_{Fp_{D'}}
			\ncline{->}{GVHD}{GM}^{Gq_D}
			\ncline{->}{GVHD'}{GM}_{Gq_{D'}}
			\ncarc[arcangle=-90]{->}{FL}{GM}^h
		\end{psmatrix}
	\]
	It follows that $(p_D,q_D)=(p_{D'},q_{D'})\circ Hd$. 
\end{proof}

\begin{lemma}\label{colimits}
	With the notions from above, let $\fD$ be a small category, and let $H:\fD\to (F \downarrow G)$. Suppose that
	\begin{enumerate}
		\item $U\circ H$ has a colimit $(L,(p_D)_{D\in\fD})$,
		\item $V\circ H$ has a colimit $(M,(q_D)_{D\in\fD})$,
		\item $(FL,(Fp_D)_{D\in\fD})$ is a colimit of $FUH$,
	\end{enumerate}
	Then there is a unique morphism $h:FL\to GM$ such that $((L,h,M),(p_D,q_D)_{d\in\fD})$ is a compatible cocone of $H$. Moreover, this cocone is a colimit of $H$.
\end{lemma}
\begin{proof}
	From Lemma~\ref{weak} it follows that there is a morphism $h:FL\to GM$ such that $((L,h,M),(p_D,q_D)_{D\in\fD})$ is a compatible cocone for $H$. The uniqueness of $h$ follows from Lemma~\ref{weak}, too. So it remains to show that $((L,h,M),(p_D,q_D)_{D\in\fD})$ is a colimit. 
	
	Let $((L',h',M'), (p'_D,q'_D)_{D\in\fD})$ be another compatible cocone for $H$. Then $(L',(p'_D)_{D\in\fD})$ is a compatible cocone for $UH$, and $(M',(q'_D)_{D\in\fD})$ is a compatible cocone for $VH$. Hence, there are unique morphisms $r:L\to L'$ and $s:M\to M'$ such that $r\circ p_D=p'_D$ and $s\circ q_D=q'_D$, for all $D\in\fD$. We will show that $(r,s):(L,h,M)\to (L',h',M')$ is the unique mediating morphism. First we need to show that it is a morphism at all: For this, we use that $(FL,(Fp_D)_{D\in\fD})$ is a colimit of $FUH$. Consider the following diagram:
	\[
		\begin{psmatrix}
		& [name=FL] FL & [name=GM] GM\\
		& [name=FL'] FL' & [name=GM'] GM'\\
		[name=FUHD]FUHD & & & [name=GVHD] GVHD
		\ncline{->}{FL}{GM}^h
		\ncline{->}{FL'}{GM'}^{h'}
		\ncline{->}{FUHD}{GVHD}^{\alpha_D}
		\ncline{->}{FL}{FL'}<{Fr}
		\ncline{->}{GM}{GM'}>{Gs}
		\ncline{->}{FUHD}{FL'}^{Fp'_D}
		\ncline{->}{GVHD}{GM'}^{Gq'_D}
		\ncarc[arcangle=30]{->}{FUHD}{FL}<{Fp_D}
		\ncarc[arcangle=-30]{->}{GVHD}{GM}>{Gq_D}
		\end{psmatrix}
	\]
	The lower quadrangle commutes, because $(p'_D,q'_D)$ is a morphism. We already saw that the two triangles commute. Note that  $(GM', (Gs\circ h\circ Fp_D)_{D\in\fD})$ is a compatible cocone for $FUH$ with the mediating morphism $Gs\circ h$. Now we compute
	\begin{align*}
		h'\circ Fr\circ Fp_D &= h'\circ Fp'_D\\
		&= Gq'_D\circ\alpha_D \\
		&= Gs\circ Gq_D\circ\alpha_D\\
		&= Gs\circ h \circ Fp_D.
	\end{align*}
	Hence, $h'\circ Fr$ is another mediating morphism and we conclude that $h'\circ Fr= Gs\circ h$ and hence $(r,s)$ is a morphism. We already noted, that the two triangles of the above given diagram commute. However, this means that $(r,s)$ is mediating. Let us show the uniqueness of $(r,s)$:
	
	Suppose that $(r',s')$ is another mediating morphism. Then $U(r',s')=r'$ is a mediating morphism between $(L, (p_D)_{D\in\fD})$ and $(L', (p'_D)_{D\in\fD})$, and $V(r',s')=s'$ is a mediating morphism between $(M, (q_D)_{D\in\fD})$ and $(M', (q'_D)_{D\in\fD})$. Hence $r=r'$ and $s=s'$.
\end{proof}

\subsection{Algebroidal categories}
The notion of algebroidal categories goes back to  Banaschewski and Herrlich \cite{BanHer76}. We need this concept in order to be able  to make use of the category-theoretic version of \Fraisse's theorem due to Droste and G\"obel \cite{DroGoe92}. We closely follow the exposition from   \cite{DroGoe92}.

Let $\lambda$ be a regular cardinal. Let us consider $\lambda$ as a category. A \emph{$\lambda$-chain} in $\fC$ is a functor from $\lambda$ to $\fC$. An object $A$ of $\fC$ is called \emph{$\lambda$-small} if whenever $(S,(f_i)_{i\in\lambda})$ is the colimit of a $\lambda$-chain $F$ in $\fC$, and $h:A\to S$, then there exists a $j\in\lambda$, and a morphism $g:A\to F(j)$ such that $h=f_j\circ g$. With $\fC_{<\lambda}$ we will denote the full subcategory of $\fC$ whose objects are all $\lambda$-small objects of $\fC$. The category $\fC$ will be called \emph{semi-$\lambda$-algebroidal} if all $\mu$-chains in $\fC_{<\lambda}$ have a colimit in $\fC$, and if every object of $\fC$ is the colimit of a $\lambda$-chain in $\fC_{<\lambda}$. Moreover, $\fC$ will be called \emph{$\lambda$-algebroidal} if 
\begin{enumerate}
	\item it is semi-$\lambda$-algebroidal, 
	\item $\fC_{<\lambda}$ contains at most $\lambda$ isomorphism classes of objects, and 
	\item between any two objects from $\fC_{<\lambda}$ there are at most $\lambda$ morphisms. 
\end{enumerate}

Let us now have a look onto $\lambda$-small objects in comma-categories. 
\begin{lemma}\label{smallthings}
	Let $\fA$, $\fB$, $\fC$ be categories,  such that $\fA$ and $\fB$ have colimits of  $\lambda$-chains and such that all morphisms of $\fB$ are monomorphisms. Let $F:\fA\to\fC$ be $\lambda$-continuous, and let $G:\fB\to\fC$ any functor that preserves monomorphisms. 
	
	Let  $(A,f,B)$ be an object of $(F \downarrow G)$ such that  $A$ is $\lambda$-small in $\fA$ and $B$ is $\lambda$-small in $\fB$. Then $(A,f,B)$ is $\lambda$-small in $(F \downarrow G)$.
\end{lemma}
\begin{proof}
	Let $H:\lambda\to (F \downarrow G)$ be a $\lambda$-chain, and let $((L,h,M),(a_i,b_i)_{i\in\lambda})$ be a colimit of $H$. Let $(a,b):(A,f,B)\to(L,h,M)$ be a morphism. Since $\fA$ and $\fB$ have colimits of $\lambda$-chains, it follows that $U\circ H$, and $V\circ H$ have colimits, and since $F$ is $\lambda$-continuous, by Lemma~\ref{colimits}, we have that $(L,(a_i)_{i\in\lambda})$ is a colimit of $U\circ H$, and $(M,(b_i)_{i\in\lambda})$ is a colimit of $V\circ H$. 
	Since $A$ is $\lambda$-small in $\fA$ and $a:A\to L$, there exists an $i\in\lambda$, and a morphism $\hat{a}:A\to UHi$ such that $a_i\circ\hat{a}=a$. Also, since $B$ is $\lambda$-small in $\fB$, and since $b:B\to M$, it follows that there exists some $j\in\lambda$ and a morphism $\hat{b}:B\to VHi$ such that $b_i\circ\hat{b}=b$. Without loss of generality, $i=j$. Consider the following diagram:
	\[
		\begin{psmatrix}
			[name=FL]FL & & & [name=GM]GM\\
			& [name=FA]FA & [name=GB]GB\\
			[name=FUHi]FUHi & & & [name=GVHi]GVHi
			\ncline{->}{FL}{GM}^h
			\ncline{->}{FA}{GB}^f
			\ncline{->}{FUHi}{GVHi}^{h_i}
			\ncline{->}{FA}{FUHi}^{F\hat{a}}
			\ncline{>->}{GB}{GVHi}^{G\hat{b}}
			\ncline{->}{FA}{FL}^{Fa}
			\ncline{>->}{GB}{GM}^{Gb}
			\ncline{->}{FUHi}{FL}<{Fa_i}
			\ncline{>->}{GVHi}{GM}<{Gb_i}
		\end{psmatrix}
	\]   
	where $Hi=(UHi,h_i,VHi)$. By the assumptions, the upper quadrangle and the two triangles of this diagram commute. We compute
	\begin{align*}Gb_i\circ h_i\circ F\hat{a} &= h\circ Fa_i\circ 	F\hat{a}\\
		&=h\circ Fa\\
		&=Gb\circ f\\
		&=Gb_i\circ G\hat{b}\circ f
	\end{align*} 
	Since $Gb_i$ is a monomorphism, we conclude that $G\hat{b}\circ f=h_i\circ F\hat{a}$ whence the whole diagram commutes. Hence $(\hat{a},\hat{b}):(A,f,B)\to Hi$ and $(a_i,b_i)\circ(\hat{a},\hat{b})=(a,b)$.
\end{proof}
Such $\lambda$-small objects $(A,f,B)$ in $(F \downarrow G)$ for which $A$ and $B$ are $\lambda$-small in $\fA$ and $\fB$, respectively, will be called \emph{inherited $\lambda$-small objects}. In principle, there may be non-inherited $\lambda$-small objects in $(F \downarrow G)$. 
\begin{lemma}\label{inherited}
	With the notions from above, if in $(F \downarrow G)$ every object is the colimit of a $\lambda$-chain of inherited $\lambda$-small objects, then every $\lambda$-small object of $(F \downarrow G)$ is inherited.
\end{lemma}
\begin{proof}
	Let $(A,f,B)$ be $\lambda$-small in $(F \downarrow G)$, let $H:\lambda\to (F \downarrow G)$ such that $Hi=(A_i,f_i,B_i)$ is inherited $\lambda$-small for all $i\in\lambda$, and such that $((A,f,B),(p_i,q_i)_{i\in\lambda})$ is a colimit of $H$. Consider the identity morphism $(1_A,1_B)$ of $(A,f,B)$. Since $(A,f,B)$ is $\lambda$-small, there is some $i\in\lambda$ and some $(a,b):(A,f,B)\to(A_i,f_i,B_i)$ such that $(1_A,1_B)=(p_i,q_i)\circ(a,b)$. In other words, $(A,f,B)$ is a retract of $(A_i,f_i,B_i)$. It follows that $A$ is a retract of $A_i$ and $B$ is a retract of $B_i$. Now it is easy to see that retracts of $\lambda$-small objects are $\lambda$-small. Hence $A$ is $\lambda$-small in $\fA$, and $B$ is $\lambda$-small in $\fB$. By Lemma~\ref{smallthings}, it follows that $(A,f,B)$ is inherited.
\end{proof}
\begin{proposition}\label{semialgebroidal}
	Let $\fA,\fB,\fC$ categories such that $\fA$ and $\fB$ are semi-$\lambda$-algebroidal, and such that all morphisms of $\fB$ are monomorphisms. Let $F:\fA\to\fC$, $G:\fB\to\fC$ be $\lambda$-continuous functors such that $F$ preserves $\lambda$-smallness and $G$ preserves monos.  Then $(F \downarrow G)$ is semi-$\lambda$-algebroidal.
\end{proposition}
\begin{proof}
	Let us first show, that every object of $(F \downarrow G)$ is the colimit of a $\lambda$-chain of inherited $\lambda$-small objects: Let $(A,f,B)\in (F \downarrow G)$. Since $\fA$ is semi-$\lambda$-algebroidal, there is a $\lambda$-chain $H$ of $\lambda$-small objects and morphisms $a_i:Hi\to A$ (for all $i\in\lambda$) such that $(A,(a_i)_{i\in\lambda})$ is a colimit of $H$. Similarly, since $\fB$ is semi-$\lambda$-algebroidal, there is a $\lambda$-chain $K$ in of $\lambda$-small objects in $\fB$ and a family of morphisms $b_i: Ki\to B$ ($i\in\lambda$), such that $(B,(b_i)_{i\in\lambda})$ is a colimit of $K$. Since $G$ is $\lambda$-continuous, we have that $(GB, (Gb_i)_{i\in\lambda})$ is a colimit of $GK$. Since $F$ preserves $\lambda$-smallness, we have that $FHi$ is $\lambda$-small. Hence, there exists a $j=j(i)$ and $b_i: FHi\to GKj(i)$ such that the following diagram commutes:
	\[
	\begin{psmatrix}
		[name=FA]FA & [name=GB]GB\\
		[name=FHi]FHi & [name=GKji]GKj(i)
		\ncline{->}{FA}{GB}^f%
		\ncline{->}{FHi}{GKji}_{h_i}%
		\ncline{->}{FHi}{FA}<{Fa_i}
		\ncline{>->}{GKji}{GB}>{Gb_{j(i)}}%
	\end{psmatrix}
	\]
	Whenever a factoring morphism of $f\circ Fa_i$ exists through $GKj$, then it exists also through $GKj'$ for all $j'>j$. Hence the function $J:\lambda\to\lambda: i\mapsto j(i)$ can be chosen to be increasing in a way that the sequence $(j(i))_{i\in\lambda}$ is cofinal in $\lambda$. By taking $K':= KJ$, we have that $\chi:=(h_i)_{i\in\lambda}$ is a natural transformation from $FH$ to $GJK$. Moreover, by cofinality, we have that $(B, (b_{j(i)})_{i\in\lambda})$ is a colimit of $JK$. By the universal property  of the comma-categories, there exists a unique functor $W:\lambda\to (F \downarrow G)$ such that $UW=H$, $VW=K'$, $\alpha*W=\chi$. It follows that $(A,f,B)$ is a colimit of $W$ and it follows from Lemma~\ref{smallthings} that $Wi$ is $\lambda$-small for all $i\in\lambda$.

	It remains to show that $(F \downarrow G)$ has colimits of all $\mu$-chains for $\mu<\lambda$. However, this is a direct consequence of Lemma~\ref{colimits}.
\end{proof}
In the proof of Proposition~\ref{semialgebroidal} we showed that every object of $(F \downarrow G)$ is the colimit of a $\lambda$-chain of inherited $\lambda$-small objects. From Lemma~\ref{inherited} it follows that under the assumptions of Proposition~\ref{semialgebroidal}, all $\lambda$-small objects of $(F \downarrow G)$ are inherited. This enables us, to formulate the following result:
\begin{proposition}\label{algebroidal}
Let $\fA,\fB,\fC$ categories such that $\fA$ and $\fB$ are $\lambda$-algebroidal, and such that all morphisms of $\fB$ are monomorphisms. Let $F:\fA\to\fC$, $G:\fB\to\fC$ be $\lambda$-continuous functors such that $F$ preserves $\lambda$-smallness and $G$ preserves monos. Additionally, suppose that for all $\lambda$-small objects $A\in\fA_{<\lambda}$, $B\in\fB_{<\lambda}$ there are at most $\lambda$ morphisms between $FA$ and $GB$. Then $(F \downarrow G)$ is $\lambda$-algebroidal.
\end{proposition}
\begin{proof}
	Proposition~\ref{semialgebroidal} we have that $(F \downarrow G)$ is semi-$\lambda$-algebroidal. We already noted, that all $\lambda$-small objects of $(F \downarrow G)$ are inherited. By this reason, the number of $\lambda$-small objects in $(F \downarrow G)$ is at most  $\lambda^3=\lambda$. Also, the number of morphisms between $\lambda$-small objects of $(F \downarrow G)$ is at most $\lambda^2=\lambda$. Hence, $(F \downarrow G)$ is $\lambda$-algebroidal. 
\end{proof}

\subsection{The Droste-G\"obel-machine}
In \cite{DroGoe92,DroGoe93}, Droste and G\"obel developed a categorical version of a classical model theoretic theorem by \Fraisse{} that characterizes universal homogeneous countable structures. This generalization is staged in $\lambda$-algebroidal categories, and we need to introduce a few more notions in order to be able to state it.

In the following, let $\fC$ be a category in which all morphisms are monomorphisms. Let $\fC^*$ be a full subcategory of $\fC$. 

Let $U\in\fC$. Then we say that 
\begin{description}
	\item[$U$ is $\fC^*$-universal] if for every $A\in\fC^*$ there is a morphism $f:A\to U$,
	\item[$U$ is $\fC^*$-homogeneous] if for every $A\in\fC^*$ and for all morphisms $f,g:A\to U$ there exists an automorphism $h$ of $U$ such that $h\circ f=g$,
\end{description}

We say that
\begin{description}
	\item[$\fC^*$ has the joint embedding property] if for all $A,B\in\fC^*$ there exists a $C\in\fC^*$ and morphisms $f:A\to C$ and $g:B\to C$,
	\item[$\fC^*$ has the amalgamation property] if for all $A$, $B$, $C$ from $\fC^*$ and $f:A\to B$, $g:A\to C$, there exists $D\in\fC^*$ and $\hat{f}:C\to D$, $\hat{g}:B\to D$ such that the following diagram commutes:
	\[
		\begin{psmatrix}
		[name=A]A & [name=B]B\\
		[name=C]C & [name=D]D.
		\ncline{->}{A}{B}^f%
		\ncline{->}{A}{C}<g%
		\ncline{->}{C}{D}^{\hat{f}}%
		\ncline{->}{B}{D}<{\hat{g}}%
		\end{psmatrix}
	\]
\end{description}

\begin{definition}
	We call a category $\fC$ a \emph{$\lambda$-amalgamation category} if
	\begin{enumerate}
		\item all morphisms of $\fC$ are monomorphisms,
		\item $\fC$ is $\lambda$-algebroidal,
		\item $\fC_{<\lambda}$ has the joint embedding property,
		\item $\fC_{<\lambda}$ has the amalgamation property.
	\end{enumerate}
\end{definition}

Let us state now the result by Droste and G\"obel:
\begin{theorem}[{\cite[Thm.1.1]{DroGoe92}}]\label{DroGoe}
	Let $\lambda$ be a regular cardinal, and let $\fC$ be a $\lambda$-algebroidal category in which all morphisms are monomorphisms. Then there exists a $\fC$-universal, $\fC_{<\lambda}$-homogeneous object in $\fC$ if and only if $\fC$ is a $\lambda$-amalgamation category. Moreover, any two $\fC$-universal, $\fC_{<\lambda}$-homogeneous objects in $\fC$ are isomorphic.
\end{theorem}

\subsection{A \Fraisse-type theorem for comma-categories}\label{xs6}
Before we can come to the formulation of a sufficient condition that the comma-category of two functors has a universal homogeneous object, we need to introduce some more  notions.

Let $\widehat\fA$ be a category and let $\fA\le\widehat\fA$ be a subcategory. We say that $\fA$ is \emph{isomorphism closed} in $\widehat\fA$ if for all $A\in\fA$ and for every isomorphism $f\in\widehat\fA(A\to B)$ we have that $B\in\fA$ and $f\in\fA(A\to B)$.

We say that $\fA$ has the \emph{free joint embedding property in $\widehat\fA$} if for all $A,B\in\fA$ there exist $C\in\fA$, and $f\in\fA(A\to C)$, and $g\in\fA(B\to C)$ such that $(C,f,g)$ is a weak coproduct in $\widehat\fA$.

We say that $\fA$ has the \emph{free amalgamation property in $\widehat\fA$} if for all $A,B,C\in\fA$ and for  all $f\in\fA(A\to B)$, $g\in\fA(A\to C)$ there exists a $D\in\fA$ and $\hat f\in\fA(C\to D)$, $\hat g\in\fA(B\to D)$ such that the following diagram is a weak pushout-square in $\widehat\fA$:
\[
\begin{psmatrix}
 B & D \\
 A & C
\ncline{->}{1,1}{1,2}^{\hat g}
\ncline{->}{2,1}{2,2}^{g}
\ncline{<-}{1,1}{2,1}<{f}%
\ncline{<-}{1,2}{2,2}<{\hat f}%
\end{psmatrix}
\]

\begin{definition}
	A pair of categories $(\fA,\widehat\fA)$ is called a $\lambda$-amalgamation pair if
	\begin{enumerate}
		\item $\fA\le\widehat\fA$  is isomorphism closed,
		\item all morphisms of $\fA$ are monomorphisms,
		\item $\fA$ is $\lambda$-algebroidal,
		\item $\fA_{<\lambda}$ has the free joint embedding property in $\widehat\fA$, and
		\item $\fA_{<\lambda}$ has the free amalgamation property in $\widehat\fA$.
	\end{enumerate}
\end{definition}

$\lambda$-amalgamation pairs capture the idea of free amalgamation classes and of strict \Fraisse-classes, that we talked about in Section~\ref{s32}.

Now we are ready to link up our previous observations in the following result:
\begin{theorem}\label{mainconstruction}
	Let $(\widehat\fA,\fA)$ be a $\lambda$-amalgamation pair, $\fB$ be a $\lambda$-amalgamation category,  and let $\fC$ be a category.
	Let $\hat F:\widehat\fA\to\fC$, $G:\fB\to\fC$ and let $F$ be the restriction of $\hat{F}$ to $\fA$. Further suppose that
	\begin{enumerate}
		\item $\hat F$ preserves weak coproducts and weak pushouts in $\fA_{<\lambda}$,
		\item $F$ and $G$ are $\lambda$-continuous, 
		\item $F$ preserves $\lambda$-smallness,
		\item $G$ preserves monomorphisms,
		\item for every $A\in\fA_{<\lambda}$ and for every $B\in \fB_{<\lambda}$ there are at most $\lambda$ morphisms in $\fC(FA\to GB)$.
	\end{enumerate}
	Then $(F \downarrow G)$ has a $(F \downarrow G)$-universal, $(F \downarrow G)_{<\lambda}$-homogeneous object. Moreover, up to isomorphism there is just one such object in $(F \downarrow G)$. 
\end{theorem}
\begin{proof}
	By construction, all morphisms of $(F \downarrow G)$ are monomorphisms. 
	From Proposition~\ref{algebroidal}, it follows that $(F \downarrow G)$ is $\lambda$-algebroidal. 
	
	Let $(A_1,f_1,B_1),(A_2,f_2,B_2)\in(F \downarrow G)_{<\lambda}$. Then $A_1,A_2\in F_{<\lambda}$, and $B_1,B_2\in\fB_{<\lambda}$. Since $\fA_{<\lambda}$ has the free joint embedding property in $\widehat\fA$, it follows that there exists a $C\in\fA_{<\lambda}$, $p_{A_1}\in\fA_{<\lambda}(A_1\to C)$, $p_{A_2}\in\fA_{<\lambda}(A_2\to C)$, such that $(C,p_{A_1},p_{A_2})$ is a weak coproduct of $A_1$ and $A_2$ in $\widehat\fA$. By assumption we have $(\hat FC,\hat Fp_{A_1},\hat Fp_{A_2})$ is a weak coproduct of $\hat FA_1$ and $\hat FA_2$ in $\fC$. 
	On the other hand, from the amalgamation property of $\fB_{<\lambda}$ it follows  that there exists $M\in\fB_{<\lambda}$ and morphisms $q_{B_1}\in\fB(B_1\to M)$ and $q_{B_2}\in\fB(B_2\to M)$. From Lemma~\ref{weak}, it follows that there exists an $h:\hat  FC\to GM$ such that $(p_{A_1},q_{B_1}):(A_1,f_1,B_1)\to(C,h,M)$ and $(p_{A_2},q_{B_2}):(A_1,f_1,B_1)\to(C,h,M)$. In other words, $(F \downarrow G)_{<\lambda}$ has the joint embedding property. 
	
	Analogously, it can be shown that $(F \downarrow G)_{<\lambda}$ has the amalgamation property. Now the existence and uniqueness of an $(F \downarrow G)$-universal, and $(F \downarrow G)_{<\lambda}$-homogeneous object in $(F \downarrow G)$ follows from Theorem~\ref{DroGoe}.
\end{proof}

\section{Tying up loose ends}

\subsection{Missing proofs from Section~\ref{s4}}\label{xs4}

\begin{proof}[proof of Theorem~\ref{hom-constraints}]
	Let $\fC:=(\overline{\cU},\to)$, $\fA:=(\overline{\cC},\injto)$, $\widehat\fA:=(\overline{\cC},\to)$ and let $\fB$ be the category that has just one object $I$ and one morphism $1_I$. Let $\hat F:\widehat\fA\to\fC$ be the identical embedding, $F$ be the restriction of $\hat F$ to $\fA$, and let $G:\fB\to\fC$ be the unique functor that maps $I$ to $\bT$. Then a routine check shows that the conditions of Theorem~\ref{mainconstruction} are fulfilled (with $\lambda=\aleph_0$). It remains to note that $(\Col_\cC(\bT),\injto)$ is isomorphic to $(F \downarrow G)$.
\end{proof}
For the proof of Theorem~\ref{hom-constraints2} we need some more preparations:

Let $G$ be a subgroup of $\Aut(\bT)$. Let $(\bA,a)$ and $(\bB,b)$ be $\bT$-colored structures in $\overline{\cC}$. A $G$-embedding is a weak embedding $(f,g)$ such that $g\in G$. 
A countable $\bT$-colored structure $(\bU,u)$ in $\overline{\cC}$ is called \emph{$G$-universal} if for every $(\bA,a)\in\Col_\bC(\bT)$ there exists a $G$-embedding from $(\bA,a)$ to $(\bU,u)$
Moreover, we call $(\bU,h)$ \emph{$G$-homogeneous} if for all finite $(\bA,a)\in\Col_\cC(\bT)$ and all $G$-embeddings $(f_1,g_1),(f_2,g_2):(\bA,a)\to(\bU,u)$ there exists a $G$-automorphism $(f_3,g_3)$ of $(\bU,u)$ such that $(f_3,g_3)\circ(f_1,g_1)=(f_2,g_2)$. 
\begin{proposition}\label{guniversal}
	With the notions from above,  let $G$ be a countable subgroup of $\Aut(\bT)$ . Then there exists a unique (up to isomorphism) $G$-universal and $G$-homogeneous structure $(\bU,u)$ in $\overline{\cC}$. 
\end{proposition}
\begin{proof}
	Let $\fC:=(\overline{\cU},\to)$, $\widehat\fA:=(\overline{\cC},\to)$, $\fA:=(\overline{\cC},\injto)$, and let  $\fB$ be the category that has just one object $\bT$ and whose morphisms are the elements of $G$. Now let $\hat F:\widehat\fA\to\fC$, $G:\fB\to\fC$ be identical embeddings, respectively. Let $F$ be the restriction of $\hat F$ to $\fA$, and let $\fD:=(F \downarrow G)$.
	 
	A routine check shows  that $(\widehat\fA,\fA)$, $\fB$, $\fC$, $\hat F$, $G$ fulfill the assumptions of Theorem~\ref{mainconstruction} (with $\lambda=\aleph_0$). Hence, $\fD$  has (up to isomorphism) a unique $\fD$-universal and $\fD_{<\aleph_0}$-homogeneous object $(\bU,u,\bT)$. 

	Clearly, $(\bU,u)$ is $G$-universal and $G$-homogeneous.
\end{proof}

\begin{proposition}\label{equal-univ}
	Let $(\bU,u)$ be a $G$-universal, $G$-homogeneous $\bT$-colored structure in $\overline\cC$ (here $G$ is an arbitrary subgroup of $\Aut(\bT)$). Then $(\bU,u)$ is also a universal, homogeneous $\bT$-colored structure. 
\end{proposition}
\begin{proof}
%
	Let $(\widehat{\bU},\hat{u})$ be a countable universal, homogeneous $\bT$-colored structure in $\overline{\cC}$. Then, there exists a $G$-embedding $(f,g):(\widehat{\bU},\hat{u})\injto (\bU,u)$. Now let $(\bA,a)$ be any countable $\bT$-colored structure in $\overline{\cC}$. Then there exists an embedding $\iota:(\bA,g^{-1}\circ a)$ into $(\widehat{\bU},\hat{u})$. But then $f\circ\iota:(\bA,g^{-1}\circ a)\injto(\bU,g^{-1}\circ u)$ is an embedding. Hence also $f\circ\iota:(\bA,a)\injto(\bU,u)$ is an embedding. Consequently, $(\bU, u)$ is universal.
	
	Let now $(\bA,a)$ be any finite $\bT$-colored structure in $\overline\cC$ and let $f_1,f_2:(\bA,a)\injto(\bU,u)$ be embeddings. Then $(f_1,1_\bT)$ and $(f_2,1_\bT)$ are $G$-embeddings. Since $(\bU,u)$ is $G$-homogeneous, there is a $G$-automorphism $(h,g)$ of $(\bU,u)$ such that $(h,g)\circ(f_1,1_\bT)=(f_2,1_\bT)$. That is, $h\circ f_1=f_2$, and $g\circ 1_\bT=1_\bT$. In other words, $g=1_\bT$ and $h$ is an automorphism of $(\bU,u)$. Hence $(\bU,u)$ is homogeneous. 
\end{proof}
A consequence of Proposition~\ref{equal-univ} is, that the construction of the universal homogeneous $\bT$-colored structure is essentially equivalent to the construction of $G$-universal, $G$-homogeneous $\bT$-colored structure in $\overline\cC$. However, the latter construction uncovers more symmetries. 

Now we are ready to prove Theorem~\ref{hom-constraints2}: 
\begin{proof}[Proof of Theorem~\ref{hom-constraints2}]
	Let $(\bU,u)$ be a universal homogeneous $\bT$-colored structure in $\overline{\cC}$.

	Let $(\bA,a)$ be a finite $\bT$-colored structure in $\overline{\cC}$, and let $(f_1,g_1), (f_2,g_2):(\bA,a)\injto(\bU,u)$ be weak embeddings. Let $G$ be the subgroup of $\Aut(\bT)$ that is generated by $g_1$ and $g_2$. Since $G$ is countable, we have, by Proposition~\ref{guniversal} that $(\bU,u)$ is $G$-homogeneous. It follows that there is a $G$-automorphism $(f,g)$ of $(\bU,u)$ such that $(f,g)\circ(f_1,g_1)=(f_2,g_2)$.
	
	Thus, $(\bU,u)$ is w-homogeneous. 
	 
	Any two countable universal w-homogeneous $\bT$-colored structures in $\overline{\cC}$ are homogeneous, by Proposition~\ref{equal-univ}. Hence, by Theorem~\ref{hom-constraints}, they are isomorphic. 
\end{proof}

\subsection{Missing proofs from Section~\ref{s3}} \label{xs3}

\begin{proof}[Proof of Theorem~\ref{univ-struc}]
	\begin{description}
		\item[About 1]From Theorem~\ref{hom-constraints}, it follows that there exists a universal $\bT$-colored structure $(\bU,u)$ in $\overline{\cC}$. So we can choose $\bU_{\cC,\bT}:=\bU$.
		\item[About 2] From Proposition~\ref{hom-const-waut}, it follows that in this case $\cAut(\bU,u)$ is oligomorphic. Since $\cAut(\bU,u)\le\Aut(\bU)$, it follows that $\Aut(\bU)$ is oligomorphic, too.
	\end{description}
	Finally, if $\bT\in\overline{\cC}$, then from Proposition~\ref{retract}, it follows that $u:\bU\to\bT$ is a retraction.
\end{proof}

\bibliographystyle{habbrv} 
\bibliography{PP11arxiv}

\end{document}